\documentclass[a4paper]{article}

\renewcommand{\tilde}{\widetilde}

\usepackage{tikz}
\usepackage{color}
\usepackage{pgfplots}
\usepackage{pgfplotstable}
\usepackage{booktabs}
\usepackage{listings}
\usepackage{amsmath,amssymb,amsfonts,amsthm}
\usetikzlibrary{patterns,pgfplots.groupplots}
\usepackage{enumerate}
\usepackage{algorithm}
\usepackage{algpseudocode}
\usepackage{ stmaryrd }
\usepackage{mathrsfs}
\usepackage{hyperref}
\usepackage{geometry}
\renewcommand{\tilde}{\widetilde}
\newcommand{\inter}[2]{\llbracket #1,#2\rrbracket}
\newcommand{\inters}[4]{\inter{#1}{#2}\times\inter{#3}{#4}}
\newcommand{\hm}{hm-toolbox}
\newcommand{\hmatrix}{HALR}
\newcommand{\HALR}{HALR}
\newcommand{\email}[1]{\texttt{#1}}
\newcommand{\hmhmatrix}{\lstinline{halr}}
\newcommand{\maxrank}{\ensuremath{\mathsf{maxrank}}}

\newcommand{\vect}{\mathrm{vec}}
\colorlet{DenseBlockColor}{gray!60}

\pgfplotstableset{
	every head row/.style={before row=\toprule,after row=\midrule},
	clear infinite
}

\DeclareMathOperator{\diag}{diag}
\newcommand{\norm}[1]{\lVert#1\rVert}
	\theoremstyle{definition}
\newtheorem{definition}{Definition}[section]
\theoremstyle{remark}
\newtheorem{remark}[definition]{Remark}

\theoremstyle{theorem}
\newtheorem{lemma}[definition]{Lemma}

\definecolor{mygreen}{RGB}{28,172,0} 
\definecolor{mylilas}{RGB}{170,55,241}
\definecolor{stringcolor}{RGB}{180,10,10}
\definecolor{mygray}{RGB}{240,240,240}
\definecolor{mygray2}{RGB}{200,200,200}
\author{
	Stefano Massei\thanks{TU Eindhoven, Netherlands,
		\email{s.massei@tue.nl}. The work of Stefano Massei has been partially supported by the SNSF research project \emph{Fast algorithms from low-rank updates}, grant number: 200020\_178806.} \and
	Leonardo Robol\thanks{Department of Mathematics, University of Pisa, 
		\email{leonardo.robol@unipi.it}. The work of Leonardo Robol was partially supported by the GNCS/INdAM project ``Metodi low-rank per problemi di algebra lineare con struttura data-sparse''.} \and 	Daniel Kressner\thanks{EPF Lausanne, Switzerland,
		\email{daniel.kressner@epfl.ch}}
}

\title{Hierarchical adaptive low-rank format with applications to discretized PDEs}

\pgfplotsset{compat=1.15}
\date{}
\begin{document}

\maketitle

\begin{abstract}
A novel compressed matrix format is proposed that combines an adaptive hierarchical partitioning of the matrix with
low-rank approximation. One typical application is the approximation of
discretized functions on rectangular domains; the flexibility of the format makes it possible to deal with functions that feature singularities in small, localized regions. To deal with time evolution and relocation of 
singularities, the partitioning can be dynamically adjusted based on 
features of the underlying data. Our format can be leveraged to efficiently solve linear systems with Kronecker 
product structure, as they arise from discretized partial differential equations (PDEs). For this purpose, these linear systems are 
rephrased as linear matrix equations and a  recursive solver is derived from low-rank updates of such equations.
We demonstrate the effectiveness of our framework
for stationary and time-dependent, linear and nonlinear PDEs, including 
the Burgers' and Allen-Cahn equations.
\end{abstract}
\bigskip

\section{Introduction}

Low-rank based data compression can sometimes lead to a dramatic acceleration of numerical simulations. A striking example is the solution of two-dimensional
elliptic PDEs on rectangular domains with smooth source terms. In this case, the (structured) discretization of the source term and the solution lead to matrices that allow for excellent low-rank approximations. Under suitable assumptions on the differential operator, one can 
recast the corresponding discretized PDE as a matrix equation \cite{palitta2016matrix,townsend2015automatic}. In turn, this yields the possibility to facilitate efficient algorithms for matrix equations with low-rank right-hand side~\cite{druskin2011analysis,benner2009adi}. However, in many 
situations of interest the smoothness property is not present in the whole domain. A typical instance are solutions that feature singularities along curves, while being highly regular elsewhere. This renders a global low-rank approximation ineffective.
Adaptive discretization schemes, such as the adaptive finite element method, are one way to handle such situations. In this work, we will focus on a purely algebraic approach.

During the last decades, there has been significant effort in developing hierarchical low-rank formats that apply low-rank approximation only locally. These formats recursively partition the matrix into blocks that are either represented as a low-rank matrix or are sufficiently small to be stored as a dense matrix. These techniques are usually applied in the context of operators with a discretization known to feature low-rank off-diagonal blocks, such as integral operators with singular kernel~\cite{borm}.
The use of these formats for representing the solution itself has also been proposed \cite{grasedyck2004existence,massei2018solving} but its applicability is limited by the fact that the 
location of the singularities needs to be known beforehand in order to define a suitable admissibility criterion \cite{hodlr, hackbusch,borm}. This makes the format too inflexible to treat time-dependent problems for which the region of non-smoothness evolves over time. In the context of tensors, it has been recently proposed a bottom-up approach to identify a partitioning of the domain and perform a piecewise compression of a target tensor by means of local high-order singular value decompositions \cite{ehrlacher2021adaptive}. A very different and promising approach proceeds by forming high-dimensional tensors from a quantization of the function and applying the so called QTT compression format; see~\cite{Kazeev2018} and the references therein.

In this paper, we propose a new format that automatically adapts the choice of the hierarchical partitioning and the location of the low-rank blocks without requiring the use of an admissibility criterion. The admissibility is decided on the fly by the success or failure
of low-rank approximation techniques. We call this format \emph{Hierarchical Adaptive Low-Rank (\HALR{})} matrices.

This work focuses on the application of \HALR{} matrices to the following class of time-dependent PDEs:
\begin{equation}\label{eq:reaction-diffusion}
\begin{cases}\frac{\partial u}{\partial t} = Lu + f(t, u, \nabla u)& (x,y)\in\Omega, \quad t\in[0, T_{\max}]\\
u(x, y, 0) = u_{0}(x,y) \quad 
\end{cases}
\end{equation}
where $\Omega\subset \mathbb R^2$ is a rectangular domain, $L$ is a linear differential operator, $f$ is nonlinear  and \eqref{eq:reaction-diffusion} is coupled with appropriate boundary conditions in space. 
Discretizing \eqref{eq:reaction-diffusion} in time  with the IMEX Euler method \cite{ascher1997implicit} and in space with, e.g., finite differences leads to 
\begin{equation}\label{eq:imex}
(I-\Delta tL_n)\mathbf u_{n,\ell+1} = \mathbf u_{n,\ell}+\Delta t(\mathbf f_{n,\ell}+ \mathbf b_{n,\ell}),
\end{equation}
where $L_n$ represents the discretization of the operator $L$, $\mathbf u_{n,\ell}$ and $\mathbf f_{n,\ell}$ are the discrete counterparts of $u$ and $f$ at  time $t_\ell:=\ell\Delta t$, $\ell\in\mathbb N$, and $\mathbf b_{n,\ell}$ accounts for the boundary conditions. When using finite differences on a tensor grid, it is natural to reshape the vectors $\mathbf u_{n,\ell}, \mathbf f_{n,\ell}, \mathbf b_{n,\ell}$  into matrices $U_{n,\ell} $, $ F_{n,\ell}$, $B_{n,\ell}$.
In our examples, the matrix $L_n$ will often take the form
$L_n = I \otimes A_{1,n} + A_{2,n} \otimes I$, a structure that is sometimes referred as 
having splitting-rank $2$ \cite{townsend2015automatic} and which allows to rephrase 
the linear system~\eqref{eq:imex} as a linear matrix equation.

As a more specific 
guiding example, let us consider the two-dimensional Burgers' equation 
over the unit square:
\begin{equation}\label{eq:burgers}
\frac{\partial u}{\partial t}= K\left(\frac{\partial^2 u}{\partial x^2} + \frac{\partial^2 u}{\partial y^2}\right) - u\cdot \left(\frac{\partial u}{\partial x}+\frac{\partial u}{\partial y}\right)\qquad
(x,y)\in \Omega:=(0,1)^2,
\end{equation}
with $K>0$.
Under suitably chosen boundary conditions, the solution of \eqref{eq:burgers} is given 
by $u(x,y,t)= \left[1+\exp\left(\frac{x+y-t}{2K}\right)\right]^{-1}$; see \cite[Example~3]{burger}. 
For a fixed time $t$, the snapshot $u_t:=u(\cdot,\cdot, t)$ describes a transition between two levels across the line $x+y=t$.  For a small coefficient $K$ the transition becomes quite sharp, see Figure~\ref{fig:burgers_sol}.  %
\begin{figure}[h!]
	\begin{center}
		~~\includegraphics
		{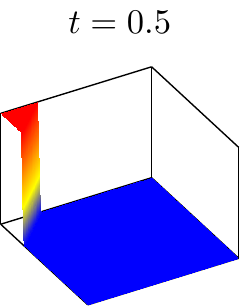}~~~~~~\includegraphics
		{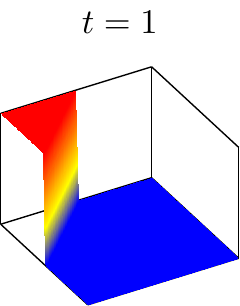}~~~~~~\includegraphics
		{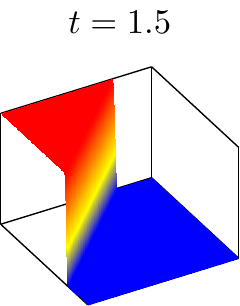}~~~~~~\includegraphics
		{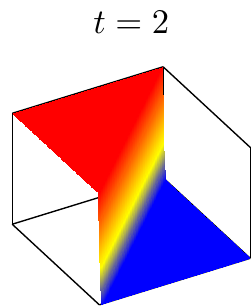} \\[8pt]
		\includegraphics[width=.2\linewidth]{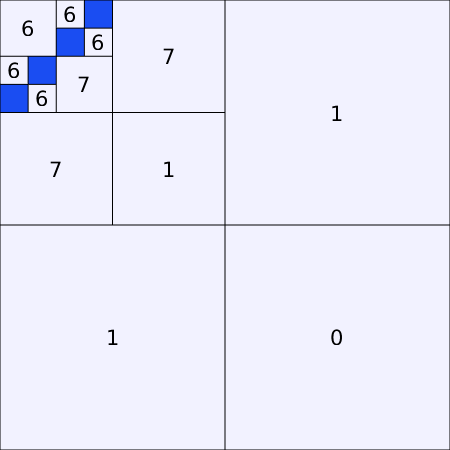}~~~~~\includegraphics[width=.2\linewidth]{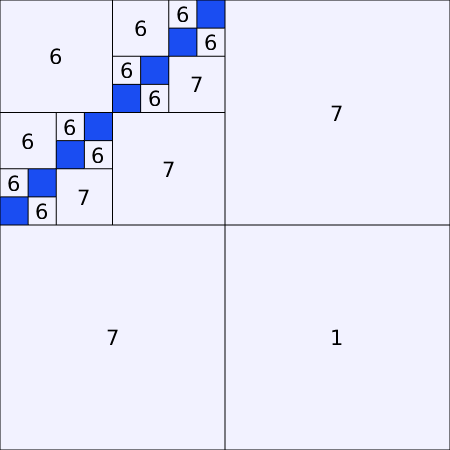}~~~~~\includegraphics[width=.2\linewidth]{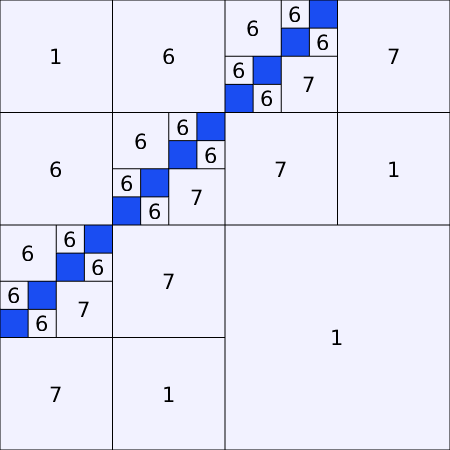}~~~~~\includegraphics[width=.2\linewidth]{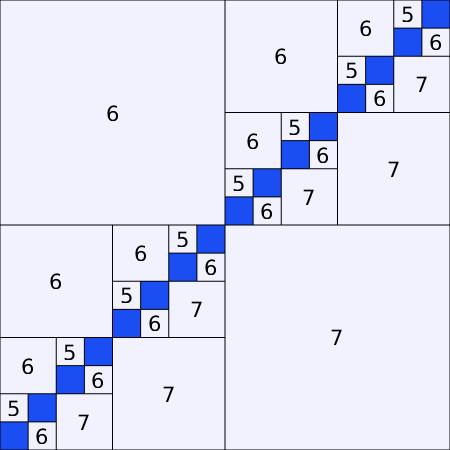}
	\end{center}
	
	\caption{Top: Snapshots  of $u(x,y,t)= \left[1 +\exp\left(\frac{x + y - t}{2K}\right)\right]^{-1}$ for $t=0.5,1,1.5,2$ and $K=0.001$. Bottom: Corresponding block low-rank structure of  $U_{n,\ell}^{\mathrm{sol}}$ for $n=4096$; the numbers indicate the rank of the corresponding block while full rank blocks are colored in blue.}
	
	\label{fig:burgers_sol}
\end{figure} %
Let $U_{n,\ell}^{\mathrm{sol}}$ be the matrix collecting the samples of $u_{t_\ell}$ on an equispaced 2D lattice; $U_{n,\ell}^{\mathrm{sol}}$
has a time dependent rank structure. More specifically, the submatrices of $U_{n,\ell}^{\mathrm{sol}}$ corresponding to subdomains which are far away
from $x+y=t_\ell$ are numerically low-rank because they 
contain samples of a smooth function over a rectangular domain; see the 
lower part of Figure~\ref{fig:burgers_sol}. 
Therefore, an efficient representation strategy  for the solution of $\eqref{eq:burgers}$ needs to adapt the block low-rank structure of $U_{n,\ell}^{\mathrm{sol}}$ according to $\ell$.

In this work, we develop techniques for:
\begin{enumerate}[(i)]
	\item Computing a \HALR{} representation for the discretization of a function explicitly given in terms of a black-box evaluation function.
	\item Solving the linear system \eqref{eq:imex} by exploiting the \HALR{} structure in the right-hand-side and the decomposition $L_n= I\otimes A_{1,n}+A_{2,n}\otimes I$. 
\end{enumerate}
Task (i) yields structured representations for the initial condition $\mathbf u_{n,0}$ and the source term $\mathbf f_{n,\ell}$. Taken together, Tasks (i) and (ii) allow to efficiently compute the matricized solution $U_{n,\ell+1}$ of~\eqref{eq:imex}. The assumption on the discretized operator in (ii)
is satisfied for $L=\frac{\partial^2}{\partial x^2}+\frac{\partial^2}{\partial y^2}$ and enables us to rephrase~\eqref{eq:imex}  as the matrix equation
\begin{equation}\label{eq:mat-eq}
\left(\frac 12I-\Delta tA_{1,n}\right)U_{n,\ell+1}+U_{n,\ell+1}\left(\frac 12I-\Delta tA_{2,n}^T\right)=\Delta tF_{n,\ell}+U_{n,\ell}+B_{n,\ell}.
\end{equation} 

The paper is organized as follows; in Section~\ref{sec:hblr} we introduce \HALR{} matrices and discuss their arithmetic.  Section~\ref{sec:lyap} focuses on solving  matrix equations of the form \eqref{eq:mat-eq} where the right-hand-side is represented in the \HALR{} format. There, we propose a divide-and-conquer method  whose cost scales comparably to the memory resources used for storing the right-hand-side.
In Section~\ref{sec:constructors} we address the problems of constructing and adapting \HALR{} representations. In particular, Section~\ref{sec:repartition} considers the following scenario: given a parameter $\mathsf{maxrank}$, determine the partitioning that provides the biggest reduction of the storage cost and uses low-rank blocks of rank bounded by $\mathsf{maxrank}$. 
In Section~\ref{sec:experiments} we incorporate \HALR{} matrices into integration schemes for PDEs and we perform numerical tests that  demonstrate the computational benefits of our approach. Conclusions are drawn in Section~\ref{sec:conclusions}.   
\subsection{Notation}
To simplify the statements of some definitions we introduce the following compact notation for intervals of consecutive integers: 
\[\inter{i_l}{i_r} :=\{ i_l, i_l+1, \dots, i_r \}\subseteq \mathbb N, \quad \text{ for } 0< i_l \leq i_r.\]
In addition, we write $\inter{i_l}{i_r} <\inter{i_l'}{i_r'} $ if $i_r<i_l'$ and we use the symbol $\sqcup$ to indicate the union of disjointed sets.

\section{HALR}\label{sec:hblr}

We are concerned with  matrix partitioning  described by  quad-trees, i.e. trees with four branches at each node. More explicitly, given a matrix $A$ we consider the block partitioning 
\begin{equation}\label{eq:partitioning}
A=\begin{bmatrix}
A_{11}&A_{12}\\
A_{21}&A_{22}
\end{bmatrix}
\end{equation}
where the blocks $A_{ij}$ can be either  dense blocks,  low-rank matrices, or recursively partitioned. 
The cases of interest are those where large portions of the matrix are in low-rank form. This is in the spirit of well established hierarchical low-rank formats such as $\mathcal H$-matrices \cite{hackbusch} and $\mathcal H^2$-matrices \cite{borm}. To formalize our deliberations, we first provide the definition of a quad-tree cluster.
\begin{definition} \label{def:quadtreecluster}
	Let $m,n\in\mathbb N$. A tree $\mathcal T$ is  called \emph{quad-tree cluster} for $\inters{1}{m}{1}{n}$ if
	\begin{itemize}
		\item the root node is $\inters{1}{m}{1}{n}$,
		\item each node $I$ is a subset of $\inters{1}{m}{1}{n}$ of the form
		\[
		I=I_r\times I_c:=\inters{m_1}{m_2}{n_1}{n_2}. 
		\]
		\item each non leaf node $I=I_r\times I_c$ has $4$ children $I_{11},I_{12},I_{21},I_{22}$, that are of the form $I_{ij}=I_{r_i}\times I_{c_j}$ such that $I_r=I_{r_1}\sqcup I_{r_2}$, $I_c=I_{c_1}\sqcup I_{c_2}$, and $I_{r_1}< I_{r_2}$, $I_{c_1}< I_{c_2}$.
		\item Each leaf node is labeled either as \texttt{dense} or \texttt{low-rank}.
	\end{itemize} 
	The \emph{depth} of $\mathcal T$ is the maximum distance of a node from the root.
\end{definition}

An example of a quad-tree cluster of depth $4$ is given in Figure~\ref{fig:cluster}; this induces the block structure of a $16\times 16$ matrix shown in the bottom part of the figure. This block structure is formalized in the following definition. 

\begin{figure}[h!]
	\begin{center}
		\centering
		\includegraphics{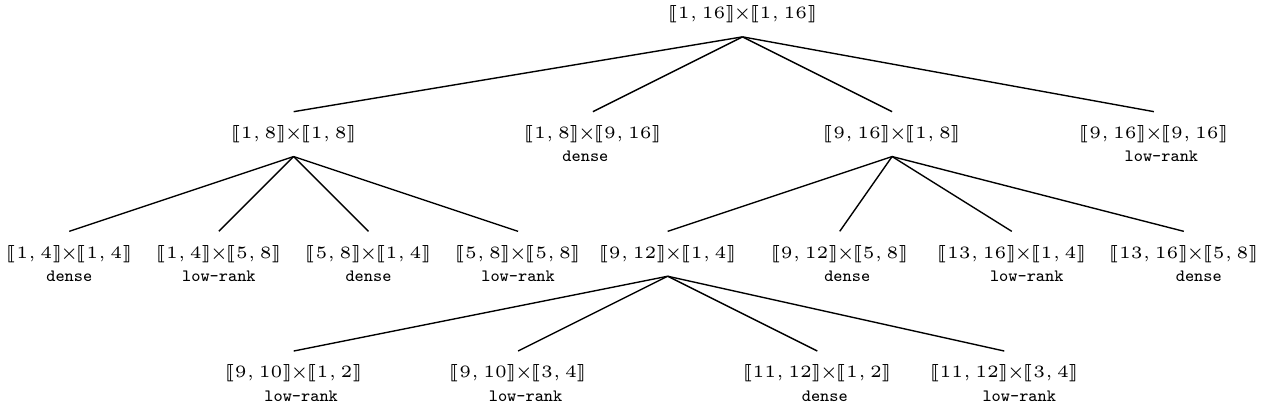}
		\centering
		\includegraphics[width=.4\textwidth]{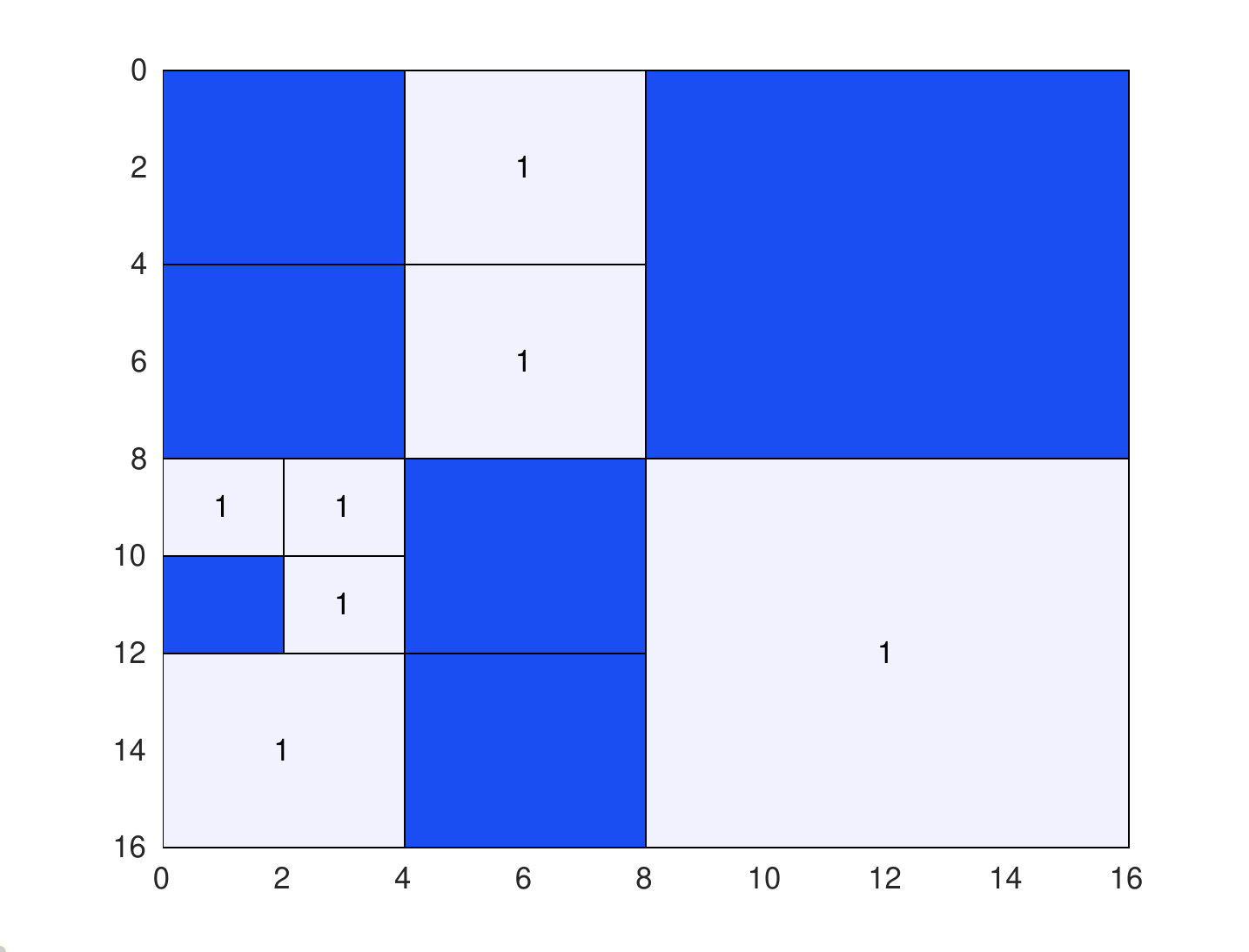}
	\end{center}
	\caption{Example of a quad-tree cluster of depth $4$ and the induced partitioning on the matrix. The leaf nodes labeled as \texttt{dense} correspond to dense blocks colored in blue. The leaf nodes labeled as \texttt{low-rank} are taken of rank $1$ and correspond to the blocks colored in gray.}
	\label{fig:cluster}
\end{figure}

\begin{definition}
	Let $A\in\mathbb C^{m\times n}$ and $\mathcal T$ be a quad-tree cluster for $\inters{1}{m}{1}{n}$.
	\begin{enumerate}
		\item Given $k\in\mathbb N$, $A$ is said to be a $(\mathcal T, k)$ Hierarchical Adaptive Low-Rank (HALR) matrix, in short $(\mathcal T, k)$-\hmatrix, if for every leaf node $I_r\times I_c$ of $\mathcal T$ labeled as \texttt{low-rank}, the submatrix $A(I_r, I_c)$ has rank at most $k$.
		\item The smallest integer $k$ for which $A$ is  $(\mathcal T, k)$-\hmatrix{} is called the \emph{$\mathcal T$-\hmatrix{} rank} of $A$.
	\end{enumerate}  
\end{definition}
There are close connections between $(\mathcal T, k)$-\HALR{} matrices and $\mathcal H$-matrices \cite{hackbusch}. More precisely, any $\mathcal H$-matrix with low-rank blocks  of rank at most $k$ and with binary row and column cluster trees is a $(\mathcal T, k)$-\HALR{} matrix. In this case the quad-tree cluster is obtained from the Cartesian product of the row and column cluster trees. The HODLR format \cite{hodlr} is a special case discussed in more detail in Section~\ref{sec:hodlr-blr}. On the other hand, Definition~\ref{def:quadtreecluster} allows to build quad-tree clusters that can not be written as subsets of any Cartesian product of a row and a column cluster trees. For instance, we might have two nodes (not having the same father) with 
column indices $I_{c},I_{c'}$ such that $I_{c}\cap I_{c'}\neq \emptyset$ and $I_{c}\not\subseteq I_{c'}$, $I_{c'}\not\subseteq I_{c}$. This makes the HALR class slightly more general than $\mathcal H$-matrices.

In the next sections we will describe operations involving HALR matrices and we will tacitly assume to have access to their structured representations, i.e. the quad-tree clusters and the low-rank factors of the \texttt{low-rank} leaves. How to retrieve the HALR representation of a given matrix will be discussed in Section~\ref{sec:constructors}.

\subsection{Matrix-vector product}

In complete analogy with the $\mathcal H$-matrix
arithmetic, the \hmatrix{} structure allows to perform the matrix-vector
product efficiently by relying on the block-recursive procedure
described in Algorithm~\ref{alg:mvec}.  

\begin{algorithm}
	\begin{algorithmic}[1]
		\Procedure{\HALR{}\_MatVec}{$A$, $v$}
		\If{$A$ is a leaf node}
		\State \Return $Av$ \Comment{Exploiting low-rank structure if present}
		\Else
		\State Partition $v = \begin{bmatrix}
		v_1 \\ v_2 
		\end{bmatrix}$
		\State \Return 
		$\begin{bmatrix}
		\Call{\HALR{}\_MatVec}{A_{11}, v_1} + \Call{\HALR{}\_MatVec}{A_{12}, v_2} \\ \Call{\HALR{}\_MatVec}{A_{21}, v_1} + \Call{\HALR{}\_MatVec}{A_{22}, v_2}
		\end{bmatrix}$
		\EndIf
		\EndProcedure
	\end{algorithmic}
	\caption{Matrix-vector product with an \hmatrix{} matrix $A$}
	\label{alg:mvec}
\end{algorithm}

In the particular case when the cluster only contains the root, $A$ itself is either \texttt{low-rank} (of rank $k$) or \texttt{dense} and  Algorithm~\ref{alg:mvec} requires $\mathcal O((m+n)k)$ and 
$\mathcal O(mn)$ flops, respectively. We remark that this cost corresponds to  
the memory required for storing $A$. This statement holds in more generality. 

\begin{lemma} \label{lem:mvec-complexity}
	Let $A \in \mathbb C^{m \times n}$ be a 
	$(\mathcal T, k)$-\hmatrix{}, and $v \in \mathbb C^n$ a vector. Computing $Av$ 
	by Algorithm~\ref{alg:mvec} requires $\mathcal O(S)$ flops, where $S$ is 
	the memory required to store $A$. 
\end{lemma}

\begin{proof}
	The result is shown by induction on the depth of $\mathcal T$. By the discussion above, the claim is true
	when $\mathcal T$ consists of a single node. If the result holds for trees of depth up 
	to $d$, and $\mathcal T$ has depth $d + 1$, the cost for $Av$ is dominated by 
	the cost of the $4$ recursive calls to \textsc{\HALR{}\_MatVec}. Using the induction assumption, it follows that the cost for these calls sums up to 
	$\mathcal O(S)$. 
\end{proof}

\subsection{Arithmetic operations}
We proceed by analyzing the interplay between the quad-tree cluster partitioning
and the usual matrix operations. If $A$ is a given 
$(\mathcal T, k)$-\hmatrix\, we can define the transpose of $\mathcal T$
as the natural cluster tree for $A^T$. 

\begin{definition}
	Let $m,n\in\mathbb N$ and $\mathcal T$ a quad-tree cluster for $\inters{1}{m}{1}{n}$. The \emph{transposed quad-tree cluster $\mathcal T^T$} is defined as the quad-tree cluster on $\inters{1}{n}{1}{m}$ obtained from $\mathcal T$ by: 
	\begin{enumerate}[(i)]
		\item replacing each node $I_r\times I_c$ with $I_c\times I_r$
		\item swapping the subtrees at $I_{12}$ and $I_{21}$ for every non-leaf node.
	\end{enumerate}
\end{definition}
Clearly, $A$ is $(\mathcal T,k)$-\hmatrix{} if and only if $A^T$ is 
$(\mathcal T^T,k)$-\hmatrix. 

\begin{remark}
	In the following, we want to regard a subtree of $\mathcal T$ again as a quad-tree cluster. For such a subtree to satisfy Definition~\ref{def:quadtreecluster}, we tacitly shift its root $\inters{m_1}{m_2}{n_1}{n_2}$
	to $\inters{1}{m_2-m_1}{1}{n_2-n_1}$ and, analogously, all other nodes in the subtree. In the opposite direction, when connecting a tree to a leaf of $\mathcal T$,  we shift the root (and the other nodes) of the tree such that it matches the index set of the leaf.
\end{remark}

Now, let us focus on arithmetic
operations between \HALR{} matrices. 
When dealing with binary operations, we need to ensure some compatibility
between the sizes of the hierarchical partitioning in order to unambiguously define the partitioning of the result. To this aim, we introduce the notions
of row and column compatibility, which will be used in the next section 
for characterizing matrix products and additions. 

\begin{definition} \label{def:cluster-compatibility}
	Given $m_A,n_A,m_B,n_B\in\mathbb N$, let $\mathcal T_A,\mathcal T_B$ be quad-tree clusters for $\inters{1}{m_A}{1}{n_A}$ and $\inters{1}{m_B}{1}{n_B}$, respectively.
	\begin{itemize}
		\item $\mathcal T_A$ and $\mathcal T_B$, with roots $I_A$ and $I_B$, are said to be \emph{row-compatible} if one of the following two conditions are satisfied:
		\begin{enumerate}[(i)]
			\item $\mathcal T_A$ or $\mathcal T_B$ only contains the root, and $m_A=m_B$.
			\item For every $i,j=1,2$ the subtrees at $(I_A)_{ij}$ and $(I_B)_{ij}$ are row-compatible.
		\end{enumerate}
		\item $\mathcal T_A$ and $\mathcal T_B$ are said \emph{column-compatible} if $\mathcal T_A^T$ and $\mathcal T_B^T$ are row-compatible.
		\item $\mathcal T_A$ and $\mathcal T_B$ are said \emph{compatible} if  they are both row- and column-compatible.
	\end{itemize}
\end{definition}
Intuitively, two quad-trees $\mathcal T_A$ and $\mathcal T_B$
are row (resp. column) compatible if taking the same path in $\mathcal T_A$ and $\mathcal T_B$ yields index sets with the same number of row (resp. column) indices. 

According to Definition~\ref{def:cluster-compatibility}, compatibility
does not depend on the labeling of the leaf nodes. Moreover, two clusters
can be compatible even if they have different depths (or contain subtrees of different depths). The following definition introduces a partial ordering
among compatible trees. This will be used to define the intersection between quad-tree clusters, which in turn allows us to
characterize the natural
partitioning of binary matrix operations involving  $A$ and $B$.

\begin{definition}\label{def:ordering}
	Let $\mathcal T_A,\mathcal T_B$ be compatible quad-tree clusters for $\inters{1}{m}{1}{n}$.
	We write $\mathcal T_A \leq \mathcal T_B$ if one of the following  conditions is satisfied
	\begin{enumerate}[(i)]
		\item $\mathcal T_A$ only contains the root labeled as \texttt{low-rank}.
		\item $\mathcal T_B$ only contains the root labeled as \texttt{dense}.
		\item For every $i,j=1,2$ the subtrees $(\mathcal T_A)_{ij}$ and $(\mathcal T_B)_{ij}$ at $(I_A)_{ij}$ and $(I_B)_{ij}$, respectively, verify $(\mathcal T_A)_{ij}\leq(\mathcal T_B)_{ij}$.
	\end{enumerate}
\end{definition}
The idea behind Definition~\ref{def:ordering} is that $\mathcal T_A \leq T_B$ implies that a $(\mathcal T_A,k)$-\hmatrix\ matrix has a stronger structure
than an $(\mathcal T_B, k)$-\hmatrix\ one. In fact, any $(\mathcal T_A, k)$-\hmatrix\ is also a $(\mathcal T_B, k)$-\hmatrix\ for all $\mathcal T_B \geq \mathcal T_A$. 
A low-rank matrix itself
corresponds to the format with the strongest structure. 
Based on this, we define the intersection between $\mathcal T_A$ and $\mathcal T_B$ as the strongest structure among the ones which are weaker than
both $\mathcal T_A$ and $\mathcal T_B$. 

\begin{definition} \label{def:intersection}
	Let $\mathcal T_A,\mathcal T_B$ be compatible quad-tree clusters for $\inters{1}{m}{1}{n}$.
	Their intersection $\mathcal T:=\mathcal T_A\cap\mathcal T_B $ is defined recursively as follows:
	\begin{enumerate}[(i)]
		\item If $\mathcal T_A$ (resp. $\mathcal T_B$) only contain the root labeled as \texttt{low-rank} then $\mathcal T_A\cap\mathcal T_B = \mathcal T_B$ (resp. $\mathcal T_A\cap\mathcal T_B = \mathcal T_A$). 
		\item  If $\mathcal T_A$ or $\mathcal T_B$ only contain the root labeled as \texttt{dense} then $\mathcal T_A\cap\mathcal T_B$ is a tree that only contains the root  labeled as \texttt{dense}.
		\item If $\mathcal T_A$ and $\mathcal T_B$ contain more than one node then their intersection is constructed by connecting the subtrees $\mathcal T_{ij}=(\mathcal T_A)_{ij}\cap(\mathcal T_B)_{ij}$, $i,j=1,2$, to the root $I=I_A=I_B$.
	\end{enumerate}
\end{definition}

\begin{remark} \label{rem:identity}
	The neutral element for the intersection
	is given by the quad-tree $\mathcal T$ only containing the root
	labeled as \texttt{low-rank}, that is, a low-rank matrix. 
\end{remark}

We now make use of the notions defined above to infer the
structure of $A+B$ from the ones of $A$ and $B$. 

\begin{lemma} \label{lem:sumstructure}
	Let $A,B\in\mathbb C^{m\times n}$ be $(\mathcal T_A,k_A)$-\hmatrix{} and $(\mathcal T_B,k_B)$-\hmatrix, respectively. If $\mathcal T_A, \mathcal T_B$ are compatible then $A+B$ is $(\mathcal T_A\cap\mathcal T_B,k_A+k_B)$-\hmatrix.
\end{lemma}
\begin{proof}
	We recall that the sum of two matrices of rank at most $k_A$ and $k_B$, respectively, has rank at most $k_A + k_B$. The statement follows from traversing the tree $T_A\cap\mathcal T_B$; for every leaf in the tree for which both submatrices of $A$ and $B$ are low rank, the resulting submatrix in $A+B$ will have rank at most $k_A+k_B$.  
\end{proof}

It is instructive to consider two special cases. First, 
if $\mathcal T_A = \mathcal T_B$, then
$A + B$ shares the same quad-tree cluster (with higher rank). Second, in view of Remark~\ref{rem:identity}, if $A$ is low rank 
then $A + B$ has the same structure as $B$, with a rank increase by (at most)
the rank of $A$.

The proof of Lemma~\ref{lem:sumstructure} suggests a recursive procedure that is summarized in Algorithm~\ref{alg:sum}.
An inductive argument analogous to the one used for Lemma~\ref{lem:mvec-complexity} shows that the complexity of Algorithm~\ref{alg:sum} is bounded by two times the cost of storing a $(\mathcal T_A\cap\mathcal T_B,k_A+k_B)$-\HALR{} matrix. Note that this estimate can be reduced by exploiting the fact that Line~\ref{line:lowrankfact} is executed at no cost by simply appending the low-rank factors of $A,B$. For example, when $A$ is a rank-$k_A$ matrix the cost reduces to $k_A$ times the number of entries in the dense blocks of $B$, which equals the storage needed for a $(\mathcal T_B,0)$-\HALR{} matrix
\begin{algorithm}
	\begin{algorithmic}[1]
		\Procedure{\HALR{}\_Sum}{$A$, $B$}
		\If{$A$ and/or $B$ are leaf nodes labeled as \texttt{dense}}
		\State \Return the dense matrix $A+B$ 
		\ElsIf{$A$ and $B$ are leaf nodes labeled as \texttt{low-rank}}
		\State \label{line:lowrankfact} \Return a low-rank factorization of $A+B$
		\Else
		\State If $A$ (resp. $B$) is a low-rank leaf, partition it according to $B$ (resp. $A$)
		\State\Return$\begin{bmatrix}\Call{\HALR{}\_Sum}{A_{11},B_{11}}&\Call{\HALR{}\_Sum}{A_{12},B_{12}}\\ \Call{\HALR{}\_Sum}{A_{21},B_{21}}&\Call{\HALR{}\_Sum}{A_{22},B_{22}}\end{bmatrix}$	
		\EndIf
		\EndProcedure
	\end{algorithmic}
	\caption{Sum of \hmatrix{} matrices}
	\label{alg:sum}
\end{algorithm}

For a matrix product
$A\cdot B$ of \HALR{} matrices, it is natural to assume that $A^T$ and $B$ are row compatible.
Assuming that $\mathcal T_A$ and $\mathcal T_B$ denote, as usual, the quad-tree clusters
of $A$ and $B$, the matrix product $A\cdot B$ stored in the \HALR{} format
is computed  with the following procedure:
\begin{enumerate}[(i)]
	\item If $\mathcal T_A$ (resp. $\mathcal T_B$) only contains the root
	labeled as \texttt{low-rank}, then the resulting tree only contains the root labeled as \texttt{low-rank} and its factorization is obtained efficiently from the one of $A$ (resp. $B$); otherwise
	\item if $\mathcal T_A$ (resp. $\mathcal T_B$) only contains the root labeled as \texttt{dense}, then $A\cdot B$ is computed in dense arithmetic and the resulting tree only contains the root labeled as \texttt{dense}; otherwise
	\item  we partition
	\[
	A = \begin{bmatrix}
	A_{11} & A_{12} \\
	A_{21} & A_{22} \\
	\end{bmatrix}, \qquad 
	B = \begin{bmatrix}
	B_{11} & B_{12} \\
	B_{21} & B_{22} \\
	\end{bmatrix}, \qquad 
	C = AB = \begin{bmatrix}
	C_{11} & C_{12} \\
	C_{21} & C_{22} \\
	\end{bmatrix},
	\]
	determine recursively 
	$A_{ik} B_{kj}$ along with their clusters $\mathcal T_{ijk}$ 
	for $i,j,k = 1,2$, and set  $C_{ij}= A_{i1}B_{1j}+A_{i2}B_{2j}$ with cluster
	$\mathcal T_{ij1} \cap \mathcal T_{ij2}$. 
\end{enumerate}
We note that it is difficult to predict a priori the quad-tree cluster
of $AB$ because  
even if $A$ and $B$ contain many low-rank blocks, the structure may be completely lost in $AB$; see the example reported in
Figure~\ref{fig:prodbad}. 
Also computing $A\cdot B$ may cost significantly  more than the storage cost of the outcome, e.g., when $A$ and $B$ are dense matrices.
On the other hand, in Section~\ref{sec:hodlr-blr} we will show that if one of the two
factors happens to be a HODLR matrix then the cost of computing  $A\cdot B$ and its quad-tree 
structure are predictable.

\begin{figure}[h!]
	\centering 
	\includegraphics[scale=.8]{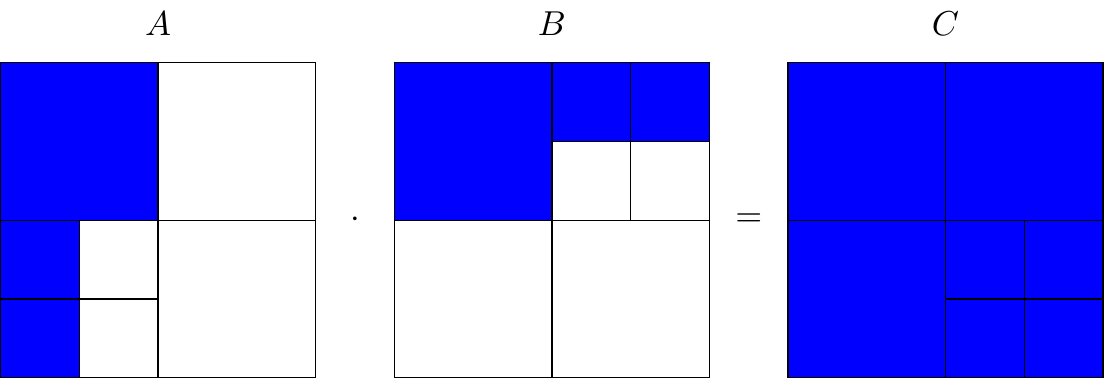}
	\caption{Example of loss of structure when computing the matrix-matrix
		multiplication. The blue region correspond to nodes labeled as \texttt{dense}, and the empty regions to nodes labeled as \texttt{low-rank}.}
	\label{fig:prodbad}
\end{figure} 

\subsection{HODLR matrices}\label{sec:hodlr-blr}

HODLR matrices are special cases of \HALR{} matrices; all the off-diagonal blocks have low rank.
To formalize this notion, we adopt the definition given
in \cite{kressner2019low}, rephrased in the formalism of quad-tree clusters. 

\begin{definition}
	A quad-tree cluster $\mathcal T_p^{(H)}$ of depth $p$ is said to be a 
	\emph{HODLR} cluster if either $p = 1$ and 
	$\mathcal T_p^{(H)}$ only contains the root labeled as \texttt{dense}, or if the children $I_{ij}$ at the root of $\mathcal T_p^{(H)}$ satisfy:
	\begin{itemize}
		\item $I_{12}$ and $I_{21}$ are leaf nodes
		labeled as \texttt{low-rank}.
		\item the subtrees at $I_{11}$ and $I_{22}$ are HODLR clusters 
		of depth $p - 1$. 
	\end{itemize}
	We say that a matrix is $(\mathcal T_p^{(H)}, k)$-HODLR if it is 
	$(\mathcal T_p^{(H)}, k)$-\hmatrix. The smallest integer $k$ for which
	a matrix $A$ is $(\mathcal T_p^{(H)}, k)$-HODLR is called the
	HODLR rank of $A$. 
\end{definition}
An example of a HODLR cluster is reported in Figure~\ref{fig:hodlrcluster}. 
A crucial property of HODLR matrices is that they are block diagonal up to a low-rank correction. This allows to predict the structure of a product of \hmatrix{} matrices whenever one of the factors is, in fact, a HODLR matrix.  

\begin{figure}[h!]
	\centering
	\includegraphics[scale=.8]{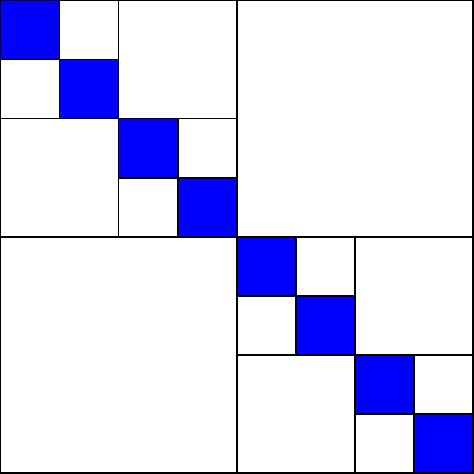}	
	\caption{Example of partitioning induced by a HODLR-cluster of depth $3$.}\label{fig:hodlrcluster}
\end{figure}

\begin{lemma}\label{lem:mat-mul}
	Let $A\in\mathbb C^{m\times m}$ be a $(\mathcal T^{(H)}_{p_A},k_A)$-HODLR matrix and $B\in\mathbb C^{m\times n}$ be a $(\mathcal T,k_B)$-\hmatrix{} matrix of depth $p_B$. If $A^T$ and $B$ are row compatible and $p_A \geq p_B$, then $A\cdot B$ is a  $(\mathcal T,k_B + (p_A-1)\cdot k_A)$-\hmatrix{} matrix.
	Similarly, if $C$ is an $n \times n$ 
	$(\mathcal T_{p_C}^{(H)}, k_C)$-HODLR matrix and $p_C \geq p_B$ then $B\cdot C$ is a 
	$(\mathcal T, k_B + (p_C-1)\cdot k_C)$-\hmatrix{} matrix
\end{lemma}
\begin{proof}
	We prove only the first statement, the second can be obtained by 
	transposition. We proceed by induction on $p_A$; if 
	$p_A = 1$, then $A$ is composed of a single dense block. Since $p_B \leq p_A$, $B$ is also composed of 
	a single block, either labeled as \texttt{low-rank} or \texttt{dense}. 
	Both structures are preserved when multiplying with $A$. 
	
	Suppose that the claim is valid for $p_A - 1 \geq 1$. If $\mathcal T$ 
	is composed of a single node, the claim is valid. Otherwise, by decomposing $A$ in its diagonal and off-diagonal parts,
	we may write
	\begin{equation}\label{eq:hodlr-split}
	\begin{bmatrix}
	A_{11} & A_{12} \\
	A_{21} & A_{22} \\
	\end{bmatrix} 	  \begin{bmatrix}
	B_{11} & B_{12} \\
	B_{21} & B_{22} \\
	\end{bmatrix} = 
	\underbrace{\begin{bmatrix}
		A_{11} B_{11} & A_{11} B_{12} \\
		A_{22} B_{21} & A_{22} B_{22} \\
		\end{bmatrix}}_{M_D} + \underbrace{\begin{bmatrix}
		A_{12} B_{21} & A_{12} B_{22} \\
		A_{21} B_{11} & A_{21} B_{12} \\
		\end{bmatrix}}_{M_O}. 
	\end{equation}
	In view of the induction step, each block of $M_D$ is a 
	$(\mathcal T_{ij}, k_B + (p_A - 2) k_A)$-\hmatrix{} matrix, for $i,j = 1,2$, 
	where $\mathcal T_{ij}$ are the quad-tree clusters associated with
	the first level of $\mathcal T$. In particular, $M_D$ is 
	$(\mathcal T, k_B + (p_A - 2) k_A)$-\hmatrix. Finally, note that all the
	blocks of $M_O$ have rank bounded by $k_A$, and therefore 
	$M_O$ is $(\mathcal T, k_A)$-\hmatrix. 
	We conclude that $AB = M_D + M_O$ is 
	$(\mathcal T, k_B + (p_A - 1) k_A)$-\hmatrix. 
\end{proof}
The complexity of multiplying by a HODLR matrix can be bounded in terms of the storage of the other factor.
\begin{lemma} \label{lemma:matrixmult}
	Under the assumptions of Lemma~\ref{lem:mat-mul} the cost of computing $A\cdot B$ is $\mathcal O(S(n_{\min}+k_A(p_A-1)))$ where $S$ is the storage cost of $B$ and $n_{\min}$ is an upper bound on the size of the dense diagonal blocks of $A$.
\end{lemma}
\begin{proof}
	For $p_A=1$, $A$ is a square matrix of size at most $n_{\min}$ while $B$ has low rank or is dense. In both cases, it directly follows that the cost of multiplication is $\mathcal O(Sn_{\min})$.
	
	For the induction step, we recall the splitting~\eqref{eq:hodlr-split} of $A\cdot B$ into two terms $M_D$ and $M_O$. The term $M_O$ is a product between $B$ and a matrix of rank (at most) $2k_A$, which, according to Lemma~\ref{lem:mvec-complexity}, requires $c_v Sk_A$ operations for some constant $c_v$. 
	The term $M_D$ consists of four products $A_{ii}B_{ij}$, where $A_{ij}$ is a HODLR matrix of depth $p_A-1$ and $B_{ij}$ is a \HALR{} matrix. 
	By induction, there is a constant $c\ge c_v+2$ such that the cost for each of these four multiplications is bounded by $c S_{ij}(n_{\min}+k_A(p_A-2))$ operations, where $S_{ij}$ denotes the storage cost of $B_{ij}$. Adding the corresponding rank-$k_A$ submatrix of $M_O$ requires at most $2  S_{ij} k_A$ operations, as discussed after Lemma~\ref{lem:sumstructure}.
	Therefore, the cost for computing the block $(i,j)$ of the product $A\cdot B$ is bounded by
	$c S_{ij}(n_{\min}+k_A(p_A-1))$.
	Summing over $i,j$ concludes the proof.
\end{proof}
\begin{remark}
	When performing arithmetic operations between HALR matrices, or HODLR and HALR matrices, it is often observed that the numerical rank of the blocks in the outcome is significantly less than the worst case scenario depicted in  Lemma~\ref{lem:sumstructure} and \ref{lemma:matrixmult}. Hence, it is advisable to perform a recompression stage, see \cite[Algorithm 2.17, p. 33]{hackbusch}, when expanding low-rank factorizations, such as in line \ref{line:lowrankfact} of Algorithm~\ref{alg:sum}. 
\end{remark}
\subsection{Solving Sylvester equations with HODLR coefficients and \hmatrix\ right-hand-side}

\label{sec:lyap}
As pointed out in the introduction, when dealing with PDEs defined 
on a rectangular two-dimensional domain, one frequently encounters linear matrix equations 
of the form
\begin{equation} \label{eq:sylvester}
AX + XB = C, 
\end{equation}
with square matrices $A,B$ and a right-hand side $C$ of matching size. To simplify the discussion we will assume that $A$ and $B$ are of equal size $n$. As $A,B$
stem from the discretization of a 1D differential 
operator, they are typically $(\mathcal T_p^{(H)}, k)$-HODLR for some small $k$.
In contrast to our previous work~\cite{kressner2019low}, where we assumed $C$ to be HODLR as well, we now consider the more general setting when $C$ is $(\mathcal T, k_C)$-\hmatrix. In the following, we require that 
$\mathcal T_p^{(H)}$ is compatible with $\mathcal T$ and $p \geq p_C$, where $p_C$ denotes the depth of $\mathcal T$. 

The particular case when $C$ is a dense matrix  will be discussed in 
further detail in Section~\ref{sec:DaC-dense}. For the moment, we let
\textsc{DenseRHS\_Sylv} denote the algorithm chosen for this case. If, instead, $C$ is low-rank, well-studied 
low-rank solvers are available, such as Krylov subspace methods and ADI (see~\cite{Simoncini2016} for a survey).
Under suitable conditions on the spectra of $A$ and $B$ and given a low-rank factorization of $C$, these solvers
return an approximation to the solution $X$ in factorized low-rank format. 
Since the specific choice of the low-rank solver is not crucial
for the following discussion, we refer
to this routine as \textsc{LowRankRHS\_Sylv}.

The equation~\eqref{eq:sylvester} can be solved recursively using an extension of our divide-and-conquer approach~\cite{kressner2019low} for HODLR matrices $C$.
If $\mathcal T$ only contains the root and, hence, $C$ is composed of a single block, we use either \textsc{DenseRHS\_Sylv} (if $C$ is \texttt{dense}) or \textsc{LowRankRHS\_Sylv}
(if $C$ is \texttt{low-rank}). 
Otherwise, we partition \eqref{eq:sylvester} according to the four children of the root of $\mathcal T$:
\[
\begin{bmatrix}
A_{11} & A_{12} \\
A_{21} & A_{22} \\
\end{bmatrix} \begin{bmatrix}
X_{11} & X_{12} \\
X_{21} & X_{22} \\
\end{bmatrix} + 
\begin{bmatrix}
X_{11} & X_{12} \\
X_{21} & X_{22} \\
\end{bmatrix} \begin{bmatrix}
B_{11} & B_{12} \\
B_{21} & B_{22} \\
\end{bmatrix} = 
\begin{bmatrix}
C_{11} & C_{12} \\
C_{21} & C_{22} \\
\end{bmatrix}. 
\]
In the spirit of \cite{kressner2019low}, we first solve the equation 
associated with the diagonal blocks of $A$ and $B$:
\begin{equation} \label{eq:diagsylv}
\begin{bmatrix}
A_{11} & 0 \\
0 & A_{22} \\
\end{bmatrix} \begin{bmatrix}
\tilde X_{11} & \tilde X_{12} \\
\tilde X_{21} & \tilde X_{22} \\
\end{bmatrix} + 
\begin{bmatrix}
\tilde X_{11} & \tilde X_{12} \\
\tilde X_{21} & \tilde X_{22} \\
\end{bmatrix} \begin{bmatrix}
B_{11} & 0 \\
0 & B_{22} \\
\end{bmatrix} = 
\begin{bmatrix}
C_{11} & C_{12} \\
C_{21} & C_{22} \\
\end{bmatrix}, 
\end{equation}
which is equivalent to solving the four decoupled equations
\begin{equation} \label{eq:Xij}
A_{ii} \tilde X_{ij} + \tilde X_{ij} B_{jj} = C_{ij}, \qquad 
i,j = 1, 2, \qquad 
\tilde X := \begin{bmatrix}
\tilde X_{11} & \tilde X_{12} \\
\tilde X_{21} & \tilde X_{22} \\
\end{bmatrix}, 
\end{equation}
where, by recursion, $\tilde X$ can be represented in the $\mathcal T$-\hmatrix\ format. 
Letting $\delta X := X - \tilde X$ and subtracting \eqref{eq:diagsylv} from
\eqref{eq:sylvester}, we obtain
\[
A \delta X + \delta X B = - \begin{bmatrix}
0 & A_{12} \\
A_{21} & 0 
\end{bmatrix} \tilde X - \tilde X \begin{bmatrix}
0 & B_{12} \\
B_{21} & 0 
\end{bmatrix}, 
\]
which is a Sylvester equation with right-hand-side of rank at most $4k$. In turn, $\delta X$ is computed using \textsc{LowRankRHS\_Sylv}, 
and $X = \tilde X + \delta X$ is retrieved 
performing a low-rank update. 
Note that the Sylvester equations in \eqref{eq:Xij} have again HODLR coefficients
and \hmatrix\ right-hand-side, with the depth decreased by one. Applying this step recursively yields the divide-and-conquer scheme 
reported in 
Algorithm~\ref{alg:dac-hblr}. 
Note that, the approximate solution returned by Algorithm~\ref{alg:dac-hblr}
retains the \hmatrix\ format, with the 
quad-tree cluster $\mathcal T$ inherited from $C$. 

\begin{algorithm}
	\begin{algorithmic}[1]
		\Procedure{D\&C\_Sylv}{$A, B, C$}
		\If{$p_C = 1$}
		\If{$C$ is \texttt{low-rank}}
		\State \Return \Call{LowRankRHS\_Sylv}{$A, B, C$} \label{lrslv1}
		\Else
		\State \Return \Call{DenseRHS\_Sylv}{$A, B, C$} \label{denseslv}
		\EndIf
		\Else
		\For{$i,j = 1, 2$}
		\State $\tilde X_{ij} \gets \Call{D\&C\_Sylv}{A_{ii}, B_{jj}, C_{ij}$}
		\EndFor
		\State $\tilde C \gets - \begin{bmatrix}
		0 & A_{12} \\
		A_{21} & 0 
		\end{bmatrix} \tilde X - \tilde X \begin{bmatrix}
		0 & B_{12} \\
		B_{21} & 0 
		\end{bmatrix}$ \label{line:lrrhs2}
		\State $\delta X \gets \Call{LowRankRHS\_Sylv}{A, B, \tilde C}$ \label{lrslv2}
		\State \Return $\begin{bmatrix}
		\tilde X_{11} & \tilde X_{12} \\
		\tilde X_{21} & \tilde X_{22} \\
		\end{bmatrix} + \delta X$ \label{xass1}
		\EndIf
		\EndProcedure
	\end{algorithmic}
	\caption{Divide-and-conquer approach for solving
		$AX + XB = C$. }
	\label{alg:dac-hblr}
\end{algorithm}

In practice, \textsc{LowRankRHS\_Sylv} in lines~\ref{lrslv1} and~\ref{lrslv2}
uses low-rank factorizations of the matrices
$C$ and $\tilde C$, and returns the solutions in factorized form.  The low-rank factors of $C$ at line~\ref{lrslv1} are given as $C$ is a leaf node of an HALR matrix. At line~\ref{lrslv1} they are easily retrieved using the low-rank factorizations of the off-diagonal blocks of $A$ and $B$ that are stored in their HODLR representations; see \cite[Section 3.1]{kressner2019low} for more details.  When $X$ is assembled by its blocks in line \ref{xass1}, 
an \hmatrix\ structure with the appropriate tree is created.

\subsubsection{Sylvester equation with dense right-hand-side}\label{sec:DaC-dense}

\begin{algorithm}
	\begin{algorithmic}[1]
		\Procedure{DenseRHS\_Sylv}{$A, B, C$}
		\If{$p_A = p_B = 1$}
		\State \Return \Call{DenseSolver\_Sylv}{$A, B, C$} \label{eq:densesylv}
		\Else
		\State Partition $C$ according to the partitioning of $A,B$:
		\State $C \gets \begin{bmatrix}
		C_{11} & C_{12} \\
		C_{21} & C_{22} \\
		\end{bmatrix} \qquad $ $C_{ij}$ \texttt{dense} for all $i,j$
		
		\For{$i,j = 1, 2$}
		\State $\tilde X_{ij} \gets \Call{DenseRHS\_Sylv}{A_{ii}, B_{jj}, C_{ij}$}
		\EndFor
		\State $\tilde C \gets - \begin{bmatrix}
		0 & A_{12} \\
		A_{21} & 0 
		\end{bmatrix} \tilde X - \tilde X \begin{bmatrix}
		0 & B_{12} \\
		B_{21} & 0 
		\end{bmatrix}$ \label{line:lrrhs}
		\State $\delta X \gets \Call{LowRankRHS\_Sylv}{A, B, \tilde C}$ \label{line:lw-sylv}
		\State \Return $\begin{bmatrix}
		\tilde X_{11} & \tilde X_{12} \\
		\tilde X_{21} & \tilde X_{22} \\
		\end{bmatrix} + \delta X$
		\EndIf
		\EndProcedure
		\caption{}
		\label{alg:densesylv}
	\end{algorithmic}	
\end{algorithm}

We now consider the solution of a Sylvester equation~\eqref{eq:sylvester} with dense $C$ and HODLR coefficients $A,B$.  This is needed in
Line~\ref{denseslv} of Algorithm~\ref{alg:dac-hblr}, but it may also be of independent interest.

For small $n$ (say, $n \leq 200$), it is most efficient to convert $A$ and $B$ to dense matrices, and use a standard  dense solver, such as the Bartels-Stewart method or RECSY~\cite{Jonsson2002}, requiring $\mathcal O(n^3)$ operations. 

For large $n$, we will see that it is more efficient to use a recursive approach instead of a dense solver. For this purpose, we partition $C$ into a block matrix in accordance with the row partition of $A$ and the column partition of $B$.
More specifically, if the size of the 
minimal blocks in the partitioning of $A$ and $B$ is $n_{\min}$ and 
$n = 2^p n_{\min}$, we represent $C$ as a $\frac{n}{n_{\min}} \times \frac{n}{n_{\min}}$ block matrix, that is, a
$(\mathcal T, 0)$-\hmatrix\ of depth $p$ with all leaf nodes labeled as \texttt{dense}. Then~\eqref{eq:sylvester} is solved recursively in analogy to Algorithm~\ref{alg:dac-hblr}. The resulting procedure is summarized in Algorithm~\ref{alg:densesylv}, where \textsc{DenseSolver\_Sylv} indicates the standard  dense solver. 

\subsubsection{Complexity analysis of the D\&C Sylvester solvers}
\label{sec:dac-complexity}

In order to perform a complexity analysis we need to make a simplifying 
assumption on the convergence of the low-rank Sylvester solver, which usually depends
on several features of the problem, such as the spectrum of $A$ and $B$.

\begin{paragraph}{\textbf{Assumption 1}}
	The computational cost of \textsc{LowRankRHS\_Sylv} for $AX + XB = C$ is 
	$\mathcal O(k_C k n \log n + k^2 n \log^2 n)$, where $n$ is the size of $A,B$, and $k$ their HODLR rank, and $k_C$ is the rank of $C$. The rank of $X$ is $O(k_C)$.
\end{paragraph} 

Assumption 1 is satisfied, for example, if the extended Krylov subspace me\-thod~\cite{Simoncini2007} con\-verges to fixed (high) accuracy in $\mathcal O(1)$ iterations and the LU factors of $A$ and $B$ are HODLR matrices of HODLR rank $\mathcal O(k)$\footnote{This is the case when $A$ and $B$ are endowed with stronger structures like hierarchical semiseparability (HSS) \cite{xia2010fast}.}. This requires the solution of a linear system with $A$ and $B$ in each iteration, via precomputing accurate approximations of the LU decompositions of $A$, $B$ at the beginning with cost $\mathcal O(k^2n\log^2n)$. In other situations, e.g., when the number of steps and/or the number of linear systems per step depend logarithmically on $n$ in order to reach a fixed accuracy, Assumption 1 and the following discussion can be easily adjusted by adding $\log n$ factors. 

Before analyzing the more general Algorithm~\ref{alg:dac-hblr}, 
it is instructive to first focus on Algorithm~\ref{alg:densesylv}. 
We note that Algorithm~\ref{alg:densesylv} solves $(\frac{n}{n_{\min}})^2$ dense 
Sylvester equations of size $n_{\min}$ and, at each level $j = 0, \ldots, p-1$, as well as $4^j$ Sylvester equations of size $\frac{n}{2^j}$ and with right-hand-sides of rank at most $4k$.
In addition, computing the low-rank factorization at line \ref{line:lrrhs} requires
$\mathcal O(\frac{n^2}{4^j} k)$ operations, amounting to a total cost of $\mathcal O(n^2)$. Under Assumption 1, \textsc{LowRankRHS\_Sylv} solves the equations at level $j$
with a cost bounded by 
$\mathcal O(\frac{k^2 n\log^2n}{2^j})$. Hence,
the total computational cost is
$\mathcal O(n^2 (n_{\min} + k^2 \log^2n))$. For large $n$
and moderate $k$, we can therefore expect that 
Algorithm~\ref{alg:densesylv} is faster than a dense solver of complexity $\mathcal O(n^3)$.

The following lemma estimates the cost of Algorithm~\ref{alg:dac-hblr} for a general \hmatrix{} matrix $C$, which reduces to our previous estimates in the two extreme cases: 
$\mathcal O(k_C k n \log n + k^2 n \log^2n )$ if $C$ is \texttt{low-rank}
and 
$\mathcal O(n^2 n_{\min} + k^2 n^2 \log^2 n)$ algorithm if $C$ is \texttt{dense}.

\begin{lemma} \label{lem:lyapcost}
	Consider the Sylvester equation $AX + XB = C$
	with $(\mathcal T_p^{(H)}, k)$-HODLR matrices $A,B \in \mathbb C^{n\times n}$ and
	a $(\mathcal T, k_C)$-\hmatrix{} matrix $C \in \mathbb C^{n\times n}$, with a quad-tree cluster $\mathcal T$ that is compatible with $\mathcal T_p^{(H)}$ and has depth $p_C \le p$.
	Suppose that $p\sim\log(n)$, 
	and let $n_{\min}$ denote the size of minimal blocks in $A, B$.
	If Assumption~1 holds, then the cost of Algorithm~\ref{alg:dac-hblr} for computing 
	the solution $X$ is $\mathcal O(S (n_{\min}  + k^2 \log^2 n))$, where
	$S$ is the storage required for $C$. 
\end{lemma}

\begin{proof}
	We prove the result by induction on $p_C$. For $p_C = 1$, the result holds following the discussion above, because $S = n^2$ if $C$ is \texttt{dense} and 
	$S = 2 k_C n$ if $C$ is \texttt{low-rank}.
	
	As the induction step is similar to the proof of Lemma~\ref{lemma:matrixmult}, we will keep it briefer. When $p_C>1$,
	Algorithm~\ref{alg:dac-hblr} consists of four stages:
	\begin{enumerate}
		\item Solution of $A_{ii} \tilde X_{ij} + \tilde X_{ij} B_{jj} = C_{ij}$, for $i,j = 1, 2$.  \\
		By the induction hypothesis, each solve is
		$\mathcal O(S_{ij} (n_{\min}  + k^2 \log^2 n))$, where $S_{ij}$ denotes the storage of $C_{ij}$ and, hence, the total cost is $\mathcal O(S (n_{\min}  + k^2 \log^2 n))$. 
		
		\item Computation of the right-hand-side in line~\ref{line:lrrhs2}. \\
		This computation involves $4k$ matrix-vector products with $\tilde X$.
		After $p-1$ recursive steps, the storage for $\tilde X$ is at most the one for $C$ plus the one for storing the $p-1$ low-rank updates, which amounts to $\mathcal O(S + kpn)$ according to Assumption 1. 
		Hence, by Lemma~\ref{lem:mvec-complexity}, the cost of this step is $\mathcal O(Sk + k^2 pn)$.  
		\item Solution of Sylvester equation in line~\ref{lrslv2}. \\
		Because the rank of the right-hand side is bounded by $4k$, the cost of this step 
		step is $\mathcal O(k^2 n \log^2 n)$.
		\item Update of $\tilde X$ in line~\ref{xass1}. \\
		The cost of this step is $\mathcal O(S k + k^2 pn)$, the storage of $\tilde X$ times the rank of $\delta X$.
	\end{enumerate}
	The total cost is dominated by the cost of Step 1, because
	one can easily prove by induction that $pn$ 
	is bounded by $\mathcal O(S)$. This completes the proof.
\end{proof}

Note that it is not the \hmatrix{} rank $k_C$ but the storage cost $S$ of the right-hand-side $C$ that appears explicitly in the complexity bound of Lemma~\ref{lem:lyapcost}. The advantage of using $S$ instead of an upper bound induced by $k_C$ is that it allows us to better explain why isolated relatively high ranks can still be treated efficiently. 

We remark that when $A$ and $B$ are banded, e.g. when they arise from the discretization of 1D differential operators,  Algorithm~\ref{alg:dac-hblr} can be executed without computing the HODLR representations of $A$ and $B$. Indeed, the low-rank factorizations of the off-diagonal blocks at line~\ref{line:lrrhs2} are easily retrieved on the fly and one can implement a solver of Sylvester equations that exploits the sparse structure of $A,B$, in \textsc{LowRankRhs\_Sylv}. 
\subsection{\HALR{} matrices in \hm}\label{sec:toolbox}

The \hm  \cite{massei2020hm} available at \url{https://github.com/numpi/hm-toolbox} is a MATLAB toolbox for working conveniently with HODLR and HSS matrices via the classes \lstinline|hodlr| and \lstinline|hss|, respectively. We have added functionality for \HALR{} matrices to the toolbox. A new class
\hmhmatrix\ has been introduced, which stores a
$(\mathcal T, k)$-\HALR{} matrix $A$ as an object with the following properties:

\begin{itemize}
	\item \lstinline|sz| is a $1 \times 2$ vector with the number of rows and columns 
	of the matrix $A$. 
	\item \lstinline|F| contains a dense representation of $A$ if it corresponds 
	to a leaf node labeled as \texttt{dense}. 
	\item \lstinline|U|, \lstinline|V| contain the low-rank factors of $A$ if it corresponds 
	to a leaf node labeled as \texttt{low-rank}. 
	\item \lstinline|admissible| is a Boolean flag that is set to 
	true for leaf nodes labeled as \texttt{low-rank}.
	\item \lstinline|A11|, \lstinline|A21|, \lstinline|A12|, \lstinline|A22| contain 
	$4$ \hmhmatrix\ objects corresponding to the children of $A$. 
\end{itemize}

The toolbox implements operations between \hmhmatrix\
objects, such as Algorithms~\ref{alg:mvec}--\ref{alg:sum} and matrix multiplication, as well as the Sylvester equation solver described in Section~\ref{sec:lyap}. Similar
to \lstinline|hodlr| and \lstinline|hss|, when arithmetic operations are performed recompression (e.g., low-rank approximation) is applied in order to limit storage while ensuring a relative accuracy. 
More specifically, an estimate of the norm of the result $C = A\ \mathrm{op}\ B$ is computed beforehand, and 
recompressions are performed using a tolerance $\epsilon \cdot \norm{C}$, 
where $\epsilon$ is a global tolerance. 
\section{Construction via adaptive detection of low-rank blocks}
\label{sec:constructors}
In this section we deal with task (i) described in the introduction, that is: given a function handle $f:\inter{1}{m}\times \inter{1}{n}\rightarrow \mathbb C$ construct an \HALR{} representation of $A=(a_{ij})_{i,j}\in\mathbb C^{m\times n}$ such that $a_{ij}=f(i,j)$. 
\subsection{Low-rank approximation}

We start by considering a simpler problem, the (global) approximation of $A$ with a low-rank matrix. Several methods have been proposed for this problem 
\cite{halko2011finding,bebendorf2000approximation,tyrtyshnikov2000incomplete,simon2000low}, which target different
scenarios. 
In the following sections we will often need to determine if a matrix is 
\emph{sufficiently low-rank} in the sense that it can be approximated, within a certain accuracy, with a matrix of rank bounded by $\maxrank$. For this purpose, we assume the availability of a procedure $(U, V, \mathrm{flag}) = \mathrm{LRA}(A, \maxrank, \epsilon)$
that returns a low-rank factorization $A \approx UV^T$, of rank at most $\maxrank$. The returned flag indicates whether the approximation verifies $\norm{A-UV^T}\lesssim \epsilon$.

In our implementation, we will rely on the \emph{adaptive cross approximation} (ACA) algorithm with partial pivoting \cite{bebendorf2000approximation}, which only requires the evaluation of a few matrix rows and columns selected by the algorithm. The parameter $\epsilon$ is used in the heuristic stopping criterion of the method, which in practice usually ensures the requirement on the absolute error stated above.
When aiming at a relative accuracy $\epsilon_{\mathsf{rel}}$, 
we need to set $\epsilon =  \epsilon_{\mathsf{rel}}\norm{A}$; if $\norm{A}$ is 
not available, it is estimated during the first ACA steps. 
The cost of ACA for returning an approximation of rank $k$ is $\mathcal O((k^2+kc_A)(m+n))$ where $c_A$ is the cost of evaluating one entry of $A$. The approximation is returned in factorized form as a product of  $m\times k$ and  $k\times n$ matrices and therefore the storage cost is $\mathcal O(k(m+n))$.

Depending on the features of $A$, other choices for the procedure \textsc{LRA} might be attractive. For instance, if the matrix-vector product by $A$ and $A^T$ can be performed 
efficiently (for instance when $A$ is sparse), then a basis for the
column range of $A$ can be well-approximated by taking matrix-vector
products with a small number of random vectors, and this can be
used to construct an approximate low-rank factorization
as described in~\cite{halko2011finding}. The methodology described in the following sections can be adapted to this context, by replacing the procedure \textsc{LRA}.
\subsection{\HALR{} approximation with prescribed cluster}
\label{sec:fixedtree}

Letting $\mathcal T$ denote a prescribed quad-tree cluster on $\inter{1}{m}\times\inter{1}{n}$, 
we consider the problem of approximating $A$ within a certain tolerance $\epsilon$, with a $(\mathcal T, k)$-\hmatrix{} $\widetilde A$ for some, hopefully small $k$. A straightforward strategy for building $\widetilde A$ is to perform the following operations on its blocks:
\begin{enumerate}[(i)]
	\item for a leaf node labeled \texttt{low-rank}, run LRA (without limitation on
	the rank)
	to  approximate the block in factored form;
	\item for a leaf node labeled \texttt{dense}, assemble and explicitly store the whole block;
	\item for a non-leaf node, proceed recursively with its children. 
\end{enumerate} 
To avoid an overestimation of the ranks for blocks of relatively small norm, 
we first approximate the norm of the entire matrix with the norm of a rough approximation of $A$ obtained by running \textsc{LRA} for a small value of $\maxrank$. 
\subsection{\HALR{} approximation with prescribed maximum rank}\label{sec:repartition}
We now discuss the problem at the heart of \HALR{}: 
Given an integer 
$\maxrank$ determine a quad-tree cluster $\mathcal T$ and an $(\mathcal T,\widetilde k)$-\HALR{} matrix $\widetilde A$ such that $\widetilde k\leq \maxrank$ and $\norm{A-\widetilde A}\leq\epsilon$. Moreover, we ideally want $\mathcal T$ to be minimal in the sense that if $\hat A$ is another $(\hat{\mathcal T}, \hat k)$-\hmatrix{} approximating $A$ within the tolerance $\epsilon$, and $\hat k\leq \maxrank$,
then $\hat{\mathcal T}\not<\mathcal T$. In this context, we consider all the trees 
(or subtrees) that contain only dense leaves to be equivalent to a single dense node. 

We propose to compute $\mathcal T$ and $\widetilde A$  with the following greedy algorithm:
\begin{enumerate}[(i)]
	\item We apply \textsc{LRA} limited by $\maxrank$ to the matrix $A$. If this is successful, as indicated by the returned flag, then $\mathcal T$ is set to a tree with a single node that is labeled \texttt{low-rank} and contains the approximation returned by \textsc{LRA}. 
	\item If \textsc{LRA} fails and the size of $A$ is smaller than a fixed parameter $n_{\min}$ then  $\mathcal T$ is set to a tree with a single node labeled as \texttt{dense} and the matrix is formed explicitly.  
	
	\item Otherwise we split $A$ in $4$ blocks of nearly equal sizes and we proceed recursively on each block. Then:
	\begin{itemize}
		\item If the $4$ blocks are all leaves labeled as \texttt{dense}, then we merge them into a single dense block. 
		\item Otherwise, we attach to the root of $\mathcal T$ the four subtrees resulting from the recursive calls.
	\end{itemize}
\end{enumerate} 
The whole procedure is summarized in Algorithm~\ref{alg:adaptive}. 

\begin{algorithm}
	\begin{algorithmic}[1]
		\Procedure{\HALR{}\_Adaptive}{$A, \maxrank, \epsilon$}			
		\State $(m,n) \gets \Call{size}{A}$
		\State $(U,V,\mathrm{flag}) \gets \Call{LRA}{A, \maxrank, \epsilon}$
		\If{LRA succeeds}
		\State $H.U \gets U$, $H.V \gets V$, 
		$H.\mathrm{admissible} = 1$
		\Else
		\If{$\min\{m,n\}\leq n_{\min}$}
		\State $H.F \gets A$, $H.\mathrm{admissible} = 0$ 
		\Else
		\State $H.\mathrm{admissible} = 0$, $m_1 \gets \lceil \frac{m}{2} \rceil, n_1 \gets \lceil \frac{n}{2} \rceil$
		\State $H.A_{11} = \Call{\HALR{}\_Adaptive}{A(1:m_1, 1:n_1), \maxrank, \epsilon}$
		\State $H.A_{21} = \Call{\HALR{}\_Adaptive}{A(m_1+1:m, 1:n_1), \maxrank, \epsilon}$
		\State $H.A_{12} = \Call{\HALR{}\_Adaptive}{A(1:m_1,n_1+1:n), \maxrank, \epsilon}$
		\State $H.A_{22} = \Call{\HALR{}\_Adaptive}{A(m_1+1:m, n_1+1:n), \maxrank, \epsilon}$
		\If{$A_{ij}$ are labeled as \texttt{dense} for $i,j = 1,2$}
		\State $H.F \gets \left[ \begin{array}{cc}
		H.A11.F & H.A12.F\\
		H.A21.F & H.A22.F \\
		\end{array} \right]$
		\State $H.A_{ij} \gets [\ ]$ \Comment{Remove the children}						
		\EndIf
		\EndIf
		\EndIf
		\EndProcedure
	\end{algorithmic}
	\caption{Approximation of a matrix $A$ using the greedy construction
		of the quad-tree cluster $\mathcal T$. The (absolute) approximation accuracy is determined by $\epsilon$.}
	\label{alg:adaptive}
\end{algorithm}

\subsection{Refining  an existing partitioning}

As operations are performed on a $(\mathcal T, k_A)$-\HALR{} matrix $A$, its low-rank properties may evolve and it can  be beneficial to readjust the tree $\mathcal T$ accordingly by making use of Algorithm~\ref{alg:adaptive}. More specifically, we refine $\mathcal T$
by performing the following steps from bottom to top:
\begin{enumerate}[(i)]
	\item Algorithm~\ref{alg:adaptive}
	with maximum rank $\maxrank$ is applied to each leaf node and the leaf node
	is replaced with the outcome. 
	\item A node with $4$ children that are \texttt{dense} leaf nodes  
	is merged into a single \texttt{dense} leaf node.
	\item For a node with $4$ children that are \texttt{low-rank} leaf nodes, we form the low-rank 
	matrix obtained by merging them. If its numerical rank is bounded 
	by $\maxrank$, we replace the node with a \texttt{low-rank} block. Otherwise, the node remains unchanged. 
\end{enumerate}

The procedure is summarized in Algorithm~\ref{alg:refine}; to decide
whether to merge four low-rank blocks in (iii), we make use of the method \textsc{CompressFactors} that computes a 
reduced truncated singular value decomposition of $UV^T$; this requires  $\mathcal O(k^2(m+n) + k^3)$ flops, where $k$ is the number of columns of $U,V$, see \cite[Algorithm~2.17, p. 33]{hackbusch}.

In the next sections, Algorithm~\ref{alg:refine} is regularly used to deal with situations where a matrix 
$B$ is obtained from operating with $\ell$ \hmatrix\ matrices $A_1, \ldots, A_\ell$ and its tree is initially set to 
the intersection of the cluster trees of $A_1, \ldots, A_\ell$.
A relevant special case
is the one where only $A_1$ is a general \hmatrix\ matrix and all the other matrices are low-rank; in this case the initial tree for $B$ is the one of $A_1$. 

\begin{algorithm}
	\begin{algorithmic}[1]
		\Procedure{RefineCluster}{$A, \maxrank, \epsilon$}
		\If{A is leaf node}
		\State $A \gets \Call{\HALR{}\_Adaptive}{A, \maxrank, \epsilon}$
		\Else
		\State $A.A_{ij} \gets \Call{RefineCluster}{A.A_{ij}, \maxrank, \epsilon}$ for $i,j = 1, 2$.
		\If{$A.A_{ij}$ are \texttt{dense} leaf nodes for $i,j = 1,2$}
		\State $A.F \gets \begin{bmatrix}
		A.A_{11}.F & A.A_{12}.F \\
		A.A_{21}.F & A.A_{22}.F \\
		\end{bmatrix}$
		\EndIf
		\If{$A.A_{ij}$ are \texttt{low-rank} leaf nodes for $i,j = 1,2$}
		\State $U \gets \begin{bmatrix}
		A.A_{11}.U & A.A_{12}.U && \\
		&& A.A_{21}.U & A.A_{22}.U
		\end{bmatrix}$
		\State $V \gets \begin{bmatrix}
		A.A_{11}.V && A.A_{21.V} \\
		& A.A_{12}.V && A.A_{22}.V \\
		\end{bmatrix}$
		\State $(U, V) \gets \Call{CompressFactors}{U, V, \epsilon}$ \Comment{\cite[Algorithm~2.17, p. 33]{hackbusch}}
		\If{$\mathrm{rank}(UV^T) \leq \maxrank$}
		\State $(A.U, A.V) \gets (U, V)$
		\State $A.\mathrm{admissible} \gets 1$
		\State $A.A_{ij} \gets [\ ]$ for $i,j =1,2$
		\EndIf
		\EndIf
		\EndIf
		\EndProcedure
	\end{algorithmic}
	\caption{}
	\label{alg:refine}
\end{algorithm}

\section{Numerical examples}
\label{sec:experiments}
In Sections~\ref{sec:hblr} and~\ref{sec:constructors} we have developed all the tools needed to implement an efficient implicit time integration  scheme for the reaction diffusion equation  \eqref{eq:reaction-diffusion}, provided that the discretization of the operator $L$ has the Kronecker sum structure $I \otimes A_n+B_n\otimes I$. In the following, we describe in detail how to put all pieces together for the representative case  of the Burgers' equation. Then we provide numerical tests for other problems that can be treated similarly.  
The experiments have been run on a server with two
Intel(R) Xeon(R) E5-2650v4 CPU with 12 cores and 24 threads each, 
running at 2.20 GHz,
using MATLAB R2017a with the Intel(R) Math Kernel Library Version 11.3.1. In all case studies, 
the relative truncation threshold has been set to $\epsilon_{\mathsf{rel}} = 10^{-8}$.

\subsection{Burgers' equation}
We consider the following  Burgers' equation \cite[Example 3]{burger} with Dirichlet boundary conditions: 
\[
\begin{cases}\displaystyle\frac{\partial u}{\partial t}= K\left(\frac{\partial^2 u}{\partial x^2} + \frac{\partial^2 u}{\partial y^2}\right) - u\cdot \left(\frac{\partial u}{\partial x}+\frac{\partial u}{\partial y}\right)=0&(x,y)\in\Omega=(0,2)\times(0,2),\\
\displaystyle u(x, y, t)= \frac{1}{1 + \text{exp}((x + y - t)/2K)}&\text{for } t=0\text{ or } (x,y)\in\partial \Omega,
\end{cases}
\]
for $K=0.001$.
We make use of a uniform finite differences discretization in space, with step $h=\frac{2}{n-1}$, combined with a Euler IMEX method for the discretization in time with step $\Delta t= 5\cdot 10^{-4}$. This yields
\begin{equation}\label{burgers-lyap}
\left(\frac 12I-\Delta tA_n\right)U_{n,\ell+1}+U_{n,\ell+1}\left(\frac 12I-\Delta tA_n\right)=U_{n,\ell}+\Delta t(F_{n,\ell}+B_{n,\ell}) ,
\end{equation}
where, denoting with $\circ$ the Hadamard (component-wise) product, we have set
\begin{align*}
F_{n,\ell}&:=U_{n,\ell}\circ \left[D_{n,\ell}U_{n,\ell}+U_{n,\ell}D_{n,\ell}^T+ (e_nv_{n,\ell}^T+  v_{n,\ell}e_n^T)/h\right], & (v_{n,\ell})_j=u(jh, 2, \ell\Delta t),\\
B_{n,\ell}&:=(e_1w_{n,\ell+1}^T + w_{n,\ell+1}e_1^T+e_nv_{n,\ell+1}^T+  v_{n,\ell+1}e_n^T)/h^2,&(w_{n,\ell})_j=u(jh, 0, \ell\Delta t)
\end{align*}
and
\[A_n=\frac{1}{h^2}\begin{bmatrix}-2&1\\ 1& -2 &\ddots\\ &\ddots&\ddots&1\\  &&1&-2\end{bmatrix}, \qquad D_{n,\ell}=\frac{1}{h}\begin{bmatrix}
-1&1\\
&\ddots&\ddots\\
&&\ddots&1\\
&&&-1
\end{bmatrix}.\]
Note that $\mathrm{rank}(B_{n,\ell})\leq 4$.
The time stepping procedure is reported in Algorithm~\ref{alg:burgers}.

\begin{algorithm}
	\begin{algorithmic}[1]
		\Procedure{Burgers\_IMEX}{$n$, $\Delta t$, $T_{\max}$}
		\State $h\gets \frac{2}{n-1}$
		\State $\left(U_{n, 0}\right)_{ij}\gets u(ih,jh,0)$\label{step:init}
		\State $t\gets 0$, $\ell\gets 0$
		\While{$t\le T_{\max}$}
		\State $F_{n,\ell}\gets U_{n,\ell}\circ \left[D_{n,\ell}U_{n,\ell}+U_{n,\ell}D_{n,\ell}^T+ (e_nv_{n,\ell}^T+  v_{n,\ell}e_n^T)/h\right]$\label{step:Fnl}
		\State $B_{n,\ell}\gets(e_1w_{n,\ell+1}^T + w_{n,\ell+1}e_1^T+e_nv_{n,\ell+1}^T+  v_{n,\ell+1}e_n^T)/h^2$\label{step:Bnl}
		\State $R\gets U_{n,\ell}+\Delta t(F_{n,\ell}+B_{n,\ell})$\label{step:R}
		\State $U_{n,\ell+1}\gets\Call{Lyap}{\frac{1}{2} I-\Delta tA_n,R}$\Comment{Solve the Lyapunov equation \eqref{burgers-lyap}}\label{step:lyap}
		\State $t\gets t+\Delta t$, \ \ $\ell\gets\ell+1$
		\EndWhile
		\EndProcedure
	\end{algorithmic}
	\caption{Time stepping procedure for the Burgers' equation}\label{alg:burgers}
\end{algorithm}
If Algorithm~\ref{alg:burgers} is executed with standard dense numerical linear algebra each time step requires $\mathcal O(n^3)$ flops and $\mathcal O(n^2)$ storage. In order to exploit the additional structure observed in Figure~\ref{fig:burgers_sol} we propose to maintain the \HALR{} representations of the matrices $F_{n,\ell}$ and $U_{n,\ell}$. In particular:
\begin{enumerate}[(i)]
	\item At line~\ref{step:init} we employ Algorithm~\ref{alg:adaptive} to retrieve a quad-tree cluster tree $\mathcal T$ and a $(\mathcal T,\widetilde k)$-\hmatrix{} representation of $U_{n,0}$. The rank $\widetilde k$ satisfies $\widetilde k\leq \maxrank$.
	\item In place of lines~\ref{step:Fnl}--\ref{step:R} we compute a $(\mathcal T, k_R)$-\hmatrix{} representation for $R$ using the algorithm described in Section~\ref{sec:fixedtree}. More specifically, we force the quad-tree cluster to be the one of $U_{n,\ell}$. We remark that an efficient handle function for evaluating the entries of $R$ is obtained by leveraging the \hmatrix{} structure of $F_{n,\ell},U_{n,\ell}$ and the low-rank structure of $B_{n,\ell}$. 
	\item We refine the cluster tree of $R$ using Algorithm~\ref{alg:refine}. During this process, the truncation is performed according to a relative threshold 
	$\epsilon_{\mathsf{rel}} = 10^{-5}$, comparable with the accuracy of the time integration method. This avoids taking into account the increase of the rank caused by the accumulation of the errors. 
	\item Since the Lyapunov equation \eqref{burgers-lyap} has HODLR coefficients and \hmatrix{} right-hand-side we employ Algorithm~\ref{alg:dac-hblr} for its solution at line~\ref{step:lyap}. Consequently, the matrix $U_{n,\ell+1}$ inherits the same quad-tree cluster of $R$.
\end{enumerate}  
Note that the refinement of the cluster at step (iii) is the only operation that can modify the quad-tree cluster used to represent the solution. 
The test has been repeated for 
$\maxrank = 25, 50, 75, 100$. 
In Table~\ref{tab:burgers-times} we report the total computational time (labeled as $T_{\mathrm{tot}}$), and the maximum 
memory consumption for storing the solution in each run, measured in MB. We also report
the total time spent solving Lyapunov equations (phase $(iv)$, labeled as $T_{\mathrm{lyap}}$)
and approximating the right-hand-side and adapting the \hmatrix\ structure (phases $(ii)$--$(iii)$, labeled as $T_{\mathrm{adapt}})$. These two phases accounts for most of
the computational cost (between $85\%$ and $90\%$); the solution of the Lyapunov equation is 
the most expensive operation. The ratio $T_{\mathrm{lyap}} / T_{\mathrm{adapt}}$ seems to grow with $n$, and is around $2$ at $n = 16384$. 

Figure~\ref{fig:burgers_time} describes in detail the case $\maxrank = 50$. 
The solution at time $t = 0$ has a low-rank structure; the region
where the shock happens is confined to the origin in 
$[0, 2]^2$. After some iterations, the shock moves causing an
increase in the rank required to approximate the solution, and 
the method switches to the \hmatrix{} structure. When the time
approaches $t = 3.75$, the solution becomes numerically low-rank again,
because the shock moves close to top-right corner of $[0, 2]^2$. 
This progression
is reported in the top part
of Figure~\ref{fig:burgers_time}, which shows the time required
for each iteration, and the structure adopted by the method. 
We remark that since the 1D Laplacian can be diagonalized via the
sine transform, Algorithm~\ref{alg:burgers} can be efficiently
implemented also without exploiting the local low-rank structure. 
In particular, the iteration cost becomes $\mathcal O(n^2\log(n))$.
In the left part of Table~\ref{tab:dense-times} we have reported 
the times required by a dense version of 
Algorithm~\ref{alg:burgers} for integrating 
\eqref{eq:burgers} where any hierarchical structure is ignored, 
and the Lyapunov equations are solved by diagonalizing the 
Laplace operator using the FFT; for this case, we have also 
reported the average time for solving the Lyapunov
equation via fast diagonalization; we see that leveraging the
\HALR{} structure makes the algorithm faster since dimension 8192. 

The bottom plot of Figure~\ref{fig:burgers_time} shows the 
absolute approximation
error in the discrete $l^2$-norm, computed comparing the numerical
solution with the true solution $u(x,y,t)= \left[1+\exp\left(\frac{x+y-t}{2K}\right)\right]^{-1}$. The error curve associated with the implementation of Algorithm~\ref{alg:burgers} in dense arithmetic matches the one reported in Figure~\ref{fig:burgers_time} confirming that the low-rank truncations have negligible effects on the computed solution.   
We remark that the displayed errors come from the discretization, 
and are not introduced by the low-rank approximations in the blocks: 
we have 
verified the computations using dense unstructured matrices, 
obtaining the same results.

\begin{figure}
	\begin{center}
		\includegraphics{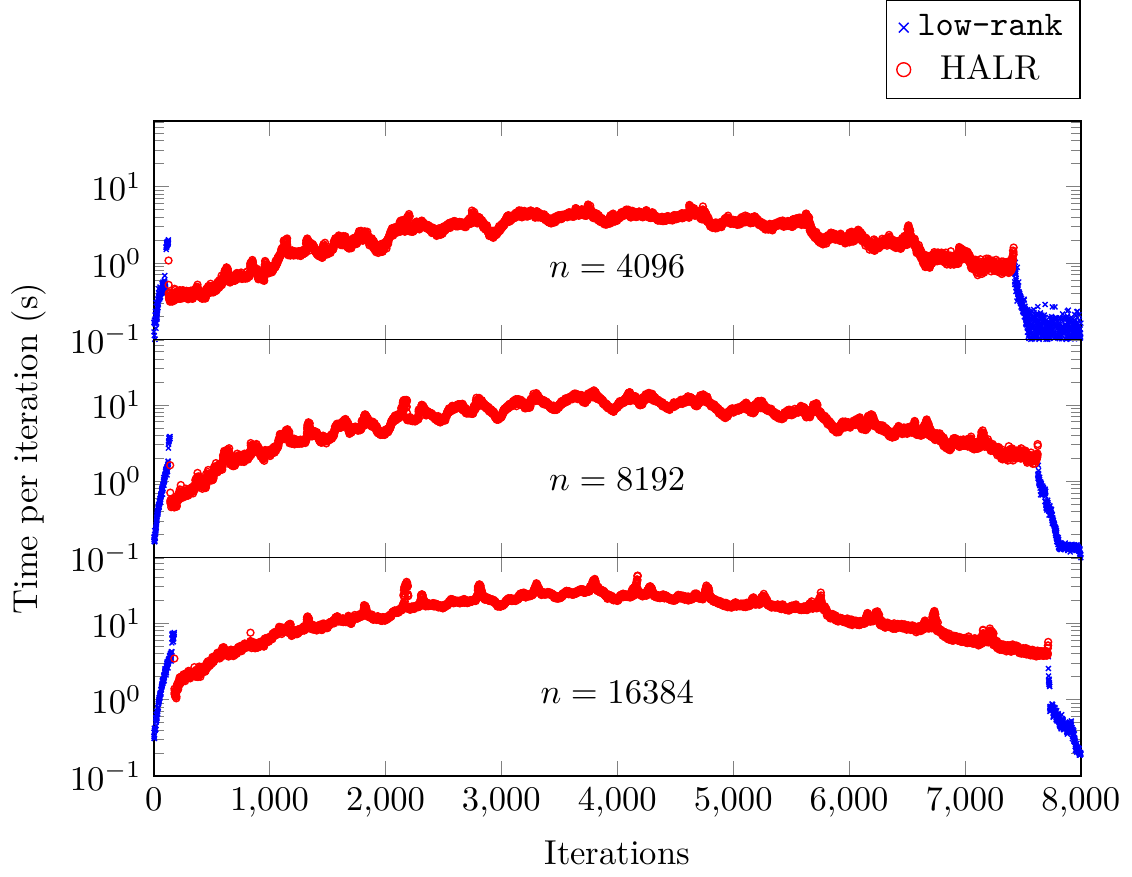} \\[8pt]
		\includegraphics{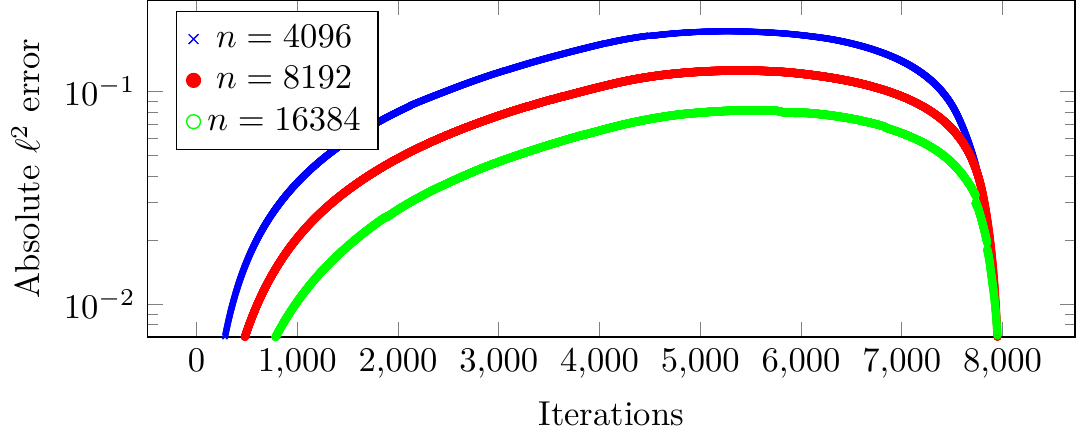}
	\end{center}
	
	\caption{Top figure: time required for solving the 
		Lyapunov equation and the adaptive approximation at each iteration; 
		the iterations marked as \texttt{low-rank} correspond to the case 
		where the matrix has the trivial partitioning with only 
		one block labeled as \texttt{low-rank}; 
		bottom figure: approximation error during the 
		iteration, obtained computing the $l^2$-norm of the difference
		with the exact solution. The reported timings are for $\maxrank=50$, and $n = 4096$, $n = 8192$, and $n=16384$.
		The reported errors are absolute; 
		for comparison, note that the $\ell^2$ norm of the solution grows monotonically from $0$ to $1$ 
		in the time interval $[0, 4]$, as the solution converges pointwise 
		to $1$. 
	}
	
	\label{fig:burgers_time}
\end{figure}

\begin{table}
	\centering
	\fbox{\parbox{\linewidth}{ 
			\begin{tabular}{c|cccc|cccc}
				& \multicolumn{4}{c|}{$\maxrank = 25$} & \multicolumn{4}{c}{$\maxrank = 50$} \\ 
				$n$ & $T_{\mathrm{tot}}$ (s) & $T_{\mathrm{lyap}}$ (s) & $T_{\mathrm{adapt}}$ (s) & Mem. & $T_{\mathrm{tot}}$ (s) & $T_{\mathrm{lyap}}$ (s) & $T_{\mathrm{adapt}}$ (s) & Mem. \\ \hline
				$4096$ & $\mathbf{20057.6}$ & $9742.6$ & $7256.7$ & $13.1$& $22334.0$ & $10604.3$ & $7767.4$ & $\mathbf{10.3}$\\ 
				$8192$ & $\mathbf{54659}$ & $29231.1$ & $17104.5$ & $17.8$& $57096.9$ & $32116.9$ & $17346.2$ & $\mathbf{16.4}$\\ 
				$16384$ & $132238.3$ & $80762.6$ & $36539.2$ & $\mathbf{25.3}$& $\mathbf{119130.4}$ & $76431.5$ & $31011.5$ & $35.3$\\
			\end{tabular}
			\\[8pt]
			\begin{tabular}{c|cccc|cccc}
				& \multicolumn{4}{c|}{$\maxrank = 75$} & \multicolumn{4}{c}{$\maxrank = 100$} \\ 
				$n$ & $T_{\mathrm{tot}}$ (s) & $T_{\mathrm{lyap}}$ (s) & $T_{\mathrm{adapt}}$ (s) & Mem. & $T_{\mathrm{tot}}$ (s) & $T_{\mathrm{lyap}}$ (s) & $T_{\mathrm{adapt}}$ (s) & Mem. \\ \hline
				$4096$ & $26727.0$ & $12915.1$ & $8923.3$ & $10.8$& $29383.2$ & $14174.7$ & $10362.5$ & $12.1$\\ 
				$8192$ & $59340.9$ & $33756.1$ & $18825.8$ & $22.4$& $63150.1$ & $34108.4$ & $22163.0$ & $24.8$\\ 
				$16384$ & $119602.0$ & $71187.1$ & $35398.9$ & $46.5$& $125688.6$ & $71050.3$ & $40701.3$ & $50.4$\\
	\end{tabular}}}
	\\[8pt]
	\caption{Time and storage required for integrating the Burgers' equation from $t = 0$ to $t = 4$, 
		for different values of $n$ and of $\maxrank$. The best times and memory usage for a given $n$ are 
		reported in bold. The reported memory is measured in Megabytes (MB), and is the maximum
		memory consumption for storing the solution during the iterations. The reported timings
		are the cumulative ones for $8000$ time steps.}
	\label{tab:burgers-times}
\end{table}

\begin{table}
	\centering
	FFT-based algorithms \\[.1cm]
	\fbox{		
		\begin{tabular}{c|cc|cc}
			& \multicolumn{2}{c|}{Burgers} & \multicolumn{2}{c}{Allen-Cahn} \\ 
			$n$ & $T_{\mathrm{tot}}$ (s) & Avg. $T_{\mathrm{lyap}}$ (s)  & $T_{\mathrm{tot}}$ (s) & Avg. $T_{\mathrm{lyap}}$ (s)  \\ \hline
			$4096$ &  $18094$ & $2.26$& $174.97$ & $0.44$ \\ 
			$8192$ & $70541$ & $8.82$&$847.3$& $2.12$   \\ 
			$16384$ & $295507$ & $36.94$ & $2967$&$7.42$ \\
	\end{tabular}}\\[8pt]
	\caption{Time required for integrating the Burgers' equation and the Allen-Cahn equation, 
		for different values of $n$, relying on sine and cosine transforms. The
		time step is chosen as in the experiments using the \hmatrix\ structure. }
	\label{tab:dense-times}
\end{table}

\subsection{Allen-Cahn equation}

The Allen-Cahn equation is a reaction-diffusion equation which describes a phase separation process. It takes the following form:
\begin{equation}\label{eq:allen-cahn}
\begin{cases}
\displaystyle \frac{\partial u}{\partial t}+\nu\left(-\Delta \right)u= g(u)& \text{in } \Omega,\\[8pt]
\displaystyle \frac{\partial u}{\partial \vec n}=0& \text{on } \partial\Omega,\\[8pt]
u(x,y,0) = u_0(x,y),
\end{cases}
\end{equation}
where $\nu=5\cdot 10^{-5}$ is the mobility, $\Omega = [0, 1]^2$
and the source term is the cubic function $g(u):= u(u - 0.5)(1 - u)$.
This test problem is described in \cite{burrage2012efficient}.
For a fixed choice of $(x,y)$, the solution $u(x,y,t)$ converges
either to $1$ or $0$ for
most points inside the domain as $t \to \infty$. 

We discretize the problem with the IMEX implicit Euler method
in time and centered finite differences in space, 
exactly as in the Burger's equation example. The only difference
is that in this problem we are considering Neumann boundary 
conditions instead of Dirichlet. 

In this example 
we choose the initial (discrete) solution randomly, 
distributed as $u(x_i, y_j, 0) \sim N(\frac 12, 1)$, with
every grid point independent of the others. Integrating the system yields 
a model for spinodal decompositions \cite{burrage2012efficient}. 
We remark that with this choice the matrix $U_{n,0}$ has 
no low-rank structure, and will be treated as a dense matrix. 
On the other hand, during the time evolution, the smoothing effect of the 
Laplacian makes the solution $U_{n,\ell}$ well-approximable by 
low-rank matrices, at least locally (see Figure~\ref{fig:allencahn-solution}). 
For even larger $\ell$, the solution converges
to either $0$ or $1$, giving rise to several ``flat regions'', 
which can be approximated by low-rank blocks, and the structure 
$U_{n,\ell}$ can be efficiently memorized using 
a $(\mathcal T, k)$-\hmatrix{} representation. 

We have integrated the solution for $t \in [0, 40]$, using 
$\Delta t = 0.1$, and grid sizes from  
$1024$ up to $16384$. The simulation has been run for 
$\maxrank = 25, 50, 75, 100$. The time and storage used for the 
integration has been reported in Table~\ref{tab:allencahn-times}, analogously
to the Burgers' equation case. Note that here the maximum memory consumption is always
attained at $t = 0$, where the solution is stored as a dense matrix. 

Figure~\ref{fig:allencahn-solution} and \ref{fig:allencahn-time} 
focus on the case $n = 4096$ and $\maxrank = 50$. 
The evolution in time of the solution and of the corresponding \hmatrix{} structure
are reported in Figure~\ref{fig:allencahn-solution}. 
\begin{figure}[h!]
	\begin{center}
		\begin{minipage}{.2\linewidth}
			\centering  $t = 0.3$ \\[8pt]
			\includegraphics[width=\linewidth]{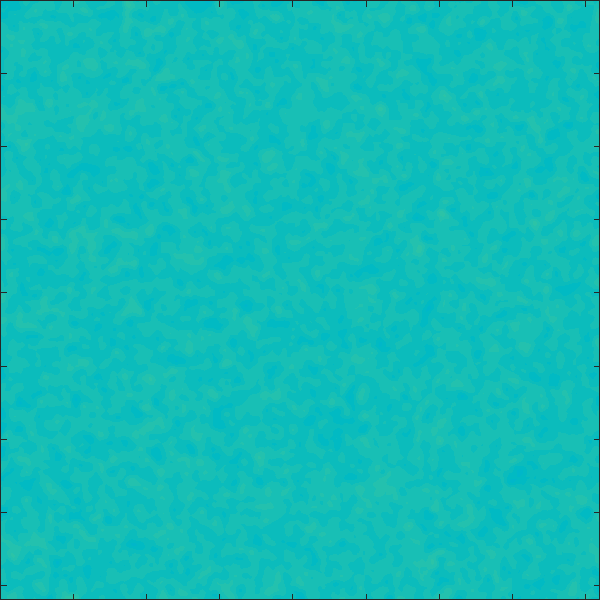}
		\end{minipage}~~~\begin{minipage}{.2\linewidth}
			\centering  $t = 0.5$ \\[8pt]
			\includegraphics[width=\linewidth]{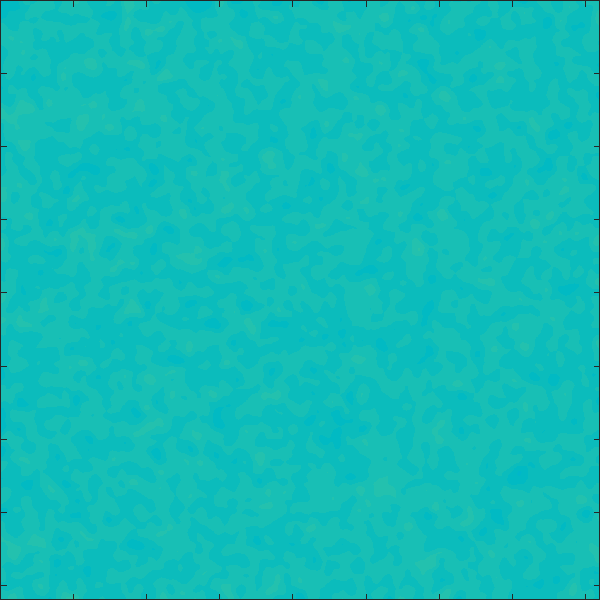}
		\end{minipage}~~~\begin{minipage}{.2\linewidth}
			\centering  $t = 18.0$ \\[8pt]
			\includegraphics[width=\linewidth]{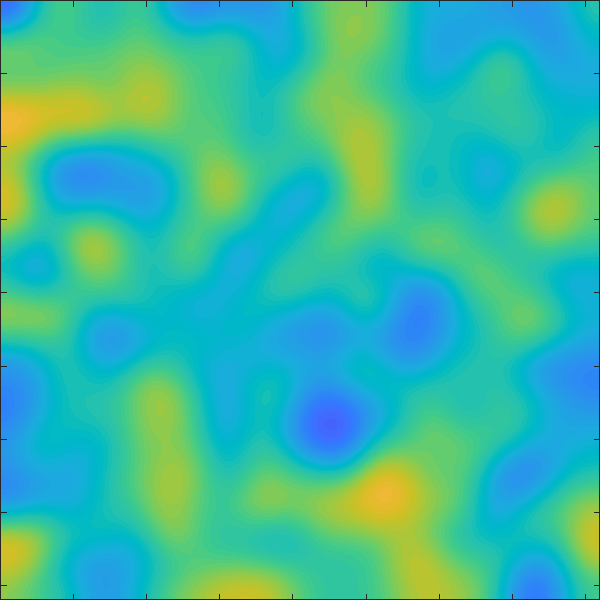}
		\end{minipage}~~~\begin{minipage}{.2\linewidth}
			\centering  $t = 35.0$ \\[8pt]
			\includegraphics[width=\linewidth]{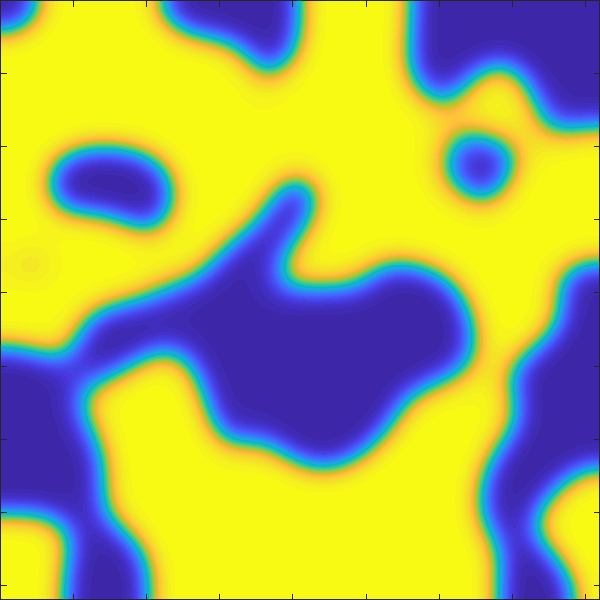}
		\end{minipage}
	\end{center}
	\bigskip 	
	\begin{center}
		\includegraphics[width=.2\linewidth]{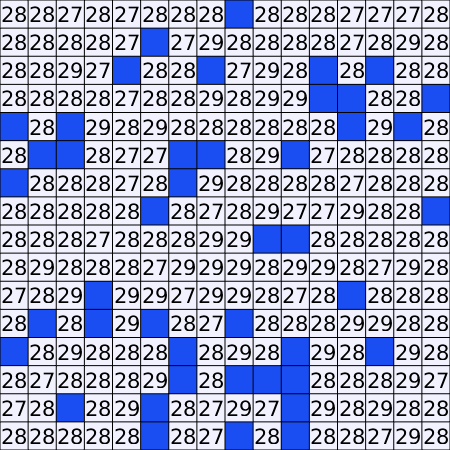}~~~\includegraphics[width=.2\linewidth]{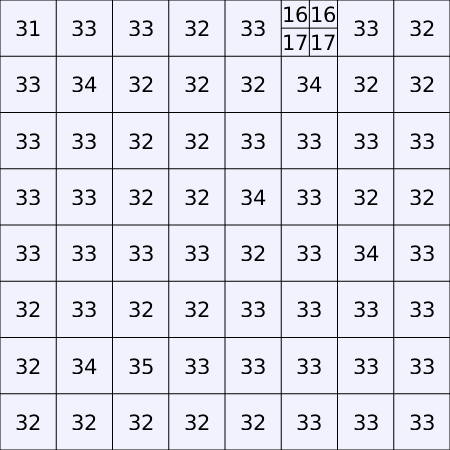}~~~\includegraphics[width=.2\linewidth]{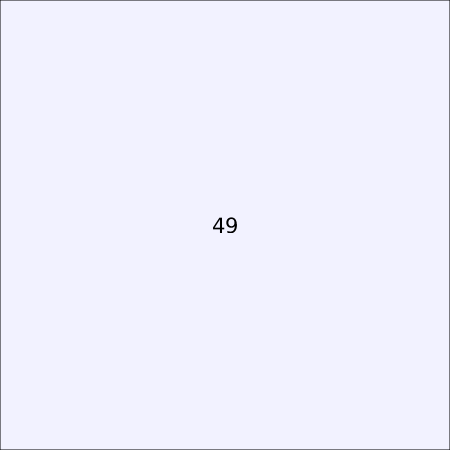}~~~\includegraphics[width=.2\linewidth]{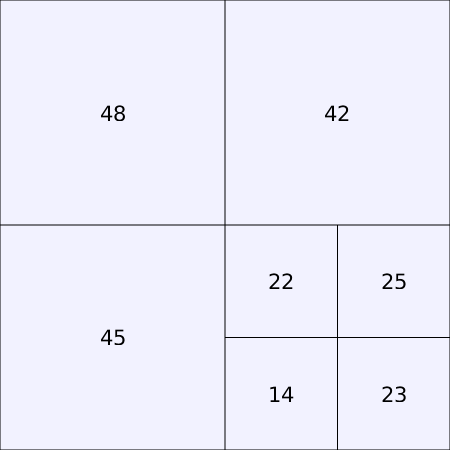}
	\end{center}
	\centering 
	\caption{Evolution of the structure and the solution at
		different time steps.}
	\label{fig:allencahn-solution}
\end{figure}
The initial structure is 
a tree with a single node labeled as \texttt{dense}, and the
\hmatrix{} representation can be used already at time $0.3$. Then,
the solution becomes (numerically) low-rank at time $6.7$ (with
rank approximately $50$);  
the third image in Figure~\ref{fig:allencahn-solution} shows 
the low-rank structure at time $t = 18$. Later, as the phase separation happens, the representation becomes again \hmatrix{}, and stabilizes
at the format shown in the fourth figure. 
The time required for each iteration, and
the structure adopted is reported in Figure~\ref{fig:allencahn-time}. Analogously to the Burgers' example, the 1D Laplacian with Neumann boundary conditions  can diagonalized via the cosine transform providing a dense method with iteration cost $\mathcal O(n^2\log(n))$. In the right part of Table~\ref{tab:dense-times}, we have reported the times required by the dense method for integrating \eqref{eq:allen-cahn} and the average time for solving the Lyapunov equation via fast diagonalization; we see that exploiting the structure makes the \HALR{} approach faster from dimension 16384.

\begin{figure}[h!]
	\begin{center}
		\includegraphics{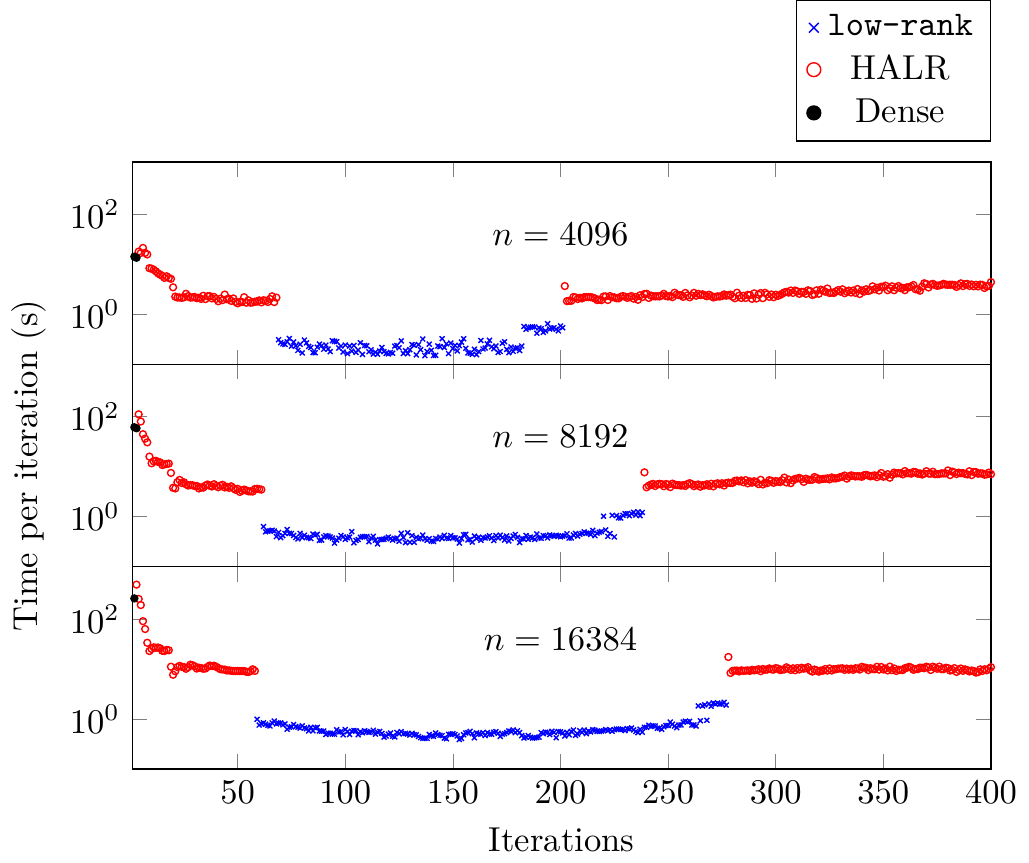}
	\end{center}
	
	\caption{Execution time per iteration for the Allen-Cahn problem. 
		The reported timings are for $\maxrank = 50$, and $n = 4096$, $n = 8192$, and $n=16384$.}
	\label{fig:allencahn-time}
\end{figure}

\begin{table}
	\centering
	\fbox{\parbox{\linewidth}{ 
			\begin{tabular}{c|cccc|cccc}
				& \multicolumn{4}{c|}{$\maxrank = 25$} & \multicolumn{4}{c}{$\maxrank = 50$} \\ 
				$n$ & $T_{\mathrm{tot}}$ (s) & $T_{\mathrm{lyap}}$ (s) & $T_{\mathrm{adapt}}$ (s) & Mem. & $T_{\mathrm{tot}}$ (s) & $T_{\mathrm{lyap}}$ (s) & $T_{\mathrm{adapt}}$ (s) & Mem. \\ \hline
				$1024$ & $277.1$ & $153.0$ & $124.1$ & $8.0$& $317.6$ & $197.1$ & $120.5$ & $8.0$\\ 
				$2048$ & $1080.6$ & $733.7$ & $346.9$ & $32.0$& $713.9$ & $561.0$ & $152.9$ & $32.0$\\ 
				$4096$ & $1754.2$ & $1479.8$ & $274.4$ & $128.0$& $900.7$ & $701.5$ & $199.2$ & $128.0$\\ 
				$8192$ & $3702.4$ & $3202.8$ & $499.5$ & $512.0$& $1823.7$ & $1428.9$ & $394.7$ & $512.0$\\ 
				$16384$ & $7704.5$ & $6392.1$ & $1312.4$ & $2048.0$& $3688.1$ & $2765.7$ & $922.5$ & $2048.0$\\ 
			\end{tabular} \\[8pt]
			\begin{tabular}{c|cccc|cccc}
				& \multicolumn{4}{c|}{$\maxrank = 75$} & \multicolumn{4}{c}{$\maxrank = 100$} \\ 
				$n$ & $T_{\mathrm{tot}}$ (s) & $T_{\mathrm{lyap}}$ (s) & $T_{\mathrm{adapt}}$ (s) & Mem. & $T_{\mathrm{tot}}$ (s) & $T_{\mathrm{lyap}}$ (s) & $T_{\mathrm{adapt}}$ (s) & Mem. \\ \hline
				$1024$ & $229.6$ & $151.1$ & $78.4$ & $8.0$& $\mathbf{187.9}$ & $118.9$ & $69.0$ & $8.0$\\ 
				$2048$ & $430.9$ & $338.2$ & $92.7$ & $32.0$& $\mathbf{346.9}$ & $252.0$ & $94.9$ & $32.0$\\ 
				$4096$ & $619.7$ & $444.2$ & $175.4$ & $128.0$& $\mathbf{505.2}$ & $325.6$ & $179.6$ & $128.0$\\ 
				$8192$ & $1424.5$ & $1036.0$ & $388.4$ & $512.0$& $\mathbf{1147.4}$ & $731.0$ & $416.3$ & $512.0$\\ 
				$16384$ & $2899.1$ & $1982.0$ & $917.1$ & $2048.0$& $\mathbf{2336.8}$ & $1331.8$ & $1005.0$ & $2048.0$\\ 	
	\end{tabular}}} \\[8pt]
	\caption{Time and storage required for integrating the Allen-Cahn equation from $0$ to $40$, 
		for different values of $n$ and of $\maxrank$. The best times for a given $N$ are 
		reported in bold. The reported memory is measured in Megabytes (MB), and is the maximum
		memory consumption for storing the solution during the iterations.}
	\label{tab:allencahn-times}
\end{table}

\subsection{Inhomogeneous Helmholtz equation}
Let us  consider the following Helmholtz equation with Neumann boundary conditions on the square $\Omega:=[-1,1]^2$:
\begin{equation}\label{eq:helm}
\begin{cases}
\Delta u + k u + f = 0\\
\displaystyle \frac{\partial u}{\partial \vec n} = 0 & \text{on } \partial \Omega
\end{cases},\qquad \left\{\begin{array}{l}
k(x,y):= 2500\cdot  e^{-50\left|x^2+(y+1)^2-\frac 14\right|}\\
f(x,y):=  \frac{e^{-x^2-y^2}}{100}
\end{array}\right..
\end{equation}
The chosen  coefficient $k(x,y)$ is negligible outside of a semi annular region centered in $(0, -1)$; the source term $f$ is concentrated around the origin. The numerical solution of \eqref{eq:helm}, reported in Figure~\ref{fig:gmres-struct}, is well approximated in the \hmatrix\ format which refines the block low rank structure in the region where $k$ takes the larger values.
\begin{figure}[h!]
	\centering 
	\includegraphics[width=.25\textwidth]{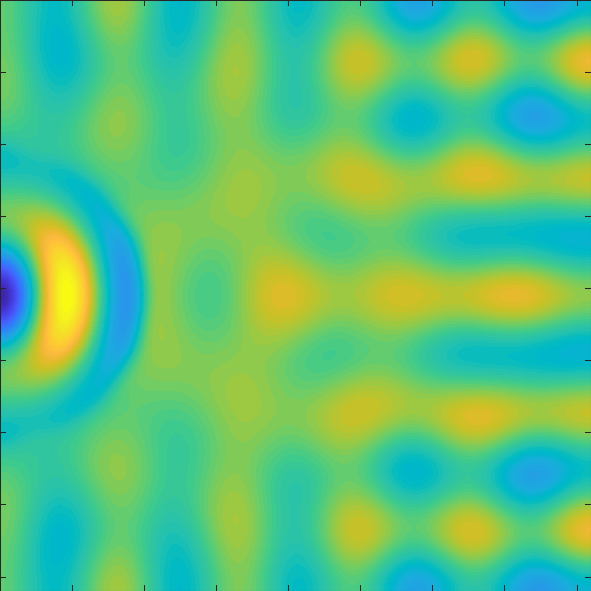}~~~~~~ \includegraphics[width=.25\textwidth]{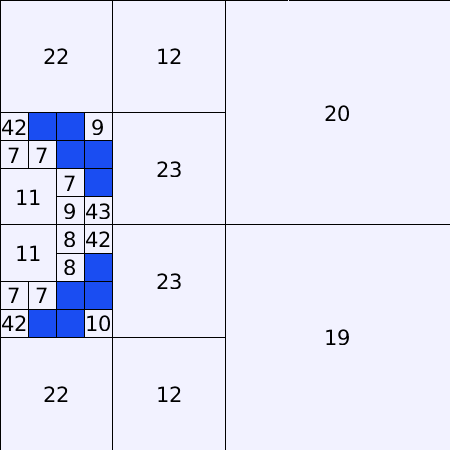}
	\caption{Solution of the Helmholtz equation \eqref{eq:helm} discretized on a $4096 \times 4096$ grid (left) and its representation in the \hmatrix\ format (right), with $\maxrank=50$.}\label{fig:gmres-struct}
\end{figure}

The usual finite difference discretization of \eqref{eq:helm} provides the $n^2\times n^2$ linear system 
\begin{equation}\label{eq:lin-sys}
\left(A\otimes I+I\otimes A+D_k\right)\vect(X) + \vect(F)=0,
\end{equation}
where $A$ is the 1D Laplacian matrix with Neumann boundary conditions, $D_k$ is the diagonal matrix containing the evaluations of $k(x,y)$ at the grid points and $F$ contains the analogous evaluations of the source term. Note that, omitting the matrix $D_k$, \eqref{eq:lin-sys} can be solved as a Lyapunov equation. In the spirit of numerical methods for generalized matrix equations \cite{benner2013low}, we propose to solve \eqref{eq:lin-sys} with a structured GMRES iteration using the Lyapunov solver as preconditioner; more specifically we store all the (matricized) vectors generated by the GMRES  in the \HALR{} format.  If necessary (when the rank grows) we adjust the partitioning of the latter via Algorithm~\ref{alg:refine}. The inner product between vectors are computed using the block recursive procedure described in Algorithm~\ref{alg:dot}, which returns the trace of $A^TB$ for two \HALR{} matrices $A$ and $B$. Finally, the solution is constructed as a linear combination of \hmatrix\ matrices. The whole procedure is reported in Algorithm~\ref{alg:pgmres}.

\begin{algorithm}
	\begin{algorithmic}[1]
		\Procedure{PGMRES}{$A$, $K$, $F$, $\mathsf{tol}$, $\maxrank$}\Comment{$\diag(\vect(K))=D_k$}
		\State $B\gets\Call{Lyap}{A, F}$
		\State $U_1\gets B/\norm{B}_F$
		\For{$j=1,2,\dots$}
		\State $R=AU_{j}+U_{j}A^T+ K\circ U_{j}$
		\State $R\gets\Call{Repartition}{R,\maxrank}$
		\State $W\gets\Call{Lyap}{A,R}$
		\For{$s=1,\dots,j$}
		\State $H_{s,j}\gets \Call{Dot}{W, U_s}$\Comment{Trace of $W^TU_s$, see Algorithm~\ref{alg:dot}}
		\State $W\gets W- H_{s, j}\cdot U_s$
		\EndFor
		\State $H_{j+1, j}\gets \norm{W}_F$, $U_{j+1}\gets W/\norm{W}_F$
		\State $y\gets \norm{B}_F H^\dagger e_1$
		\If{$\norm{Hy-\norm{B}_F e_1}<\mathsf{tol}\cdot \norm{B}_F$}
		\State \textbf{break}
		\EndIf
		\EndFor
		\State	\Return $\sum_j y_{j}U_j$
		\EndProcedure
	\end{algorithmic}
	\caption{Structured and preconditioned GMRES iteration for \eqref{eq:helm} }\label{alg:pgmres}
\end{algorithm}
\begin{algorithm}
	\begin{algorithmic}[1]
		\Procedure{Dot}{$A$, $B$}
		\If{$A$ and $B$ are leaf nodes or at least one is \texttt{low-rank}}
		\State \Return $\mathrm{Trace}(A^TB)$\Comment{Exploiting the low-rank structure of $A$ or $B$, if any}
		\Else
		\If{$A$ is \texttt{dense} or $B$ is \texttt{dense}}
		\State Set the partitioning of $A$ and $B$ equal to the finest of the two. 
		\EndIf
		\State\Return \Call{Dot}{$A_{11}, B_{11}$} + \Call{Dot}{$A_{12}, B_{12}$}+\Call{Dot}{$A_{21}, B_{21}$}+\Call{Dot}{$A_{22}, B_{22}$}
		\EndIf
		\EndProcedure
	\end{algorithmic}
	\caption{Trace inner product for two \hmatrix\ matrices }\label{alg:dot}
\end{algorithm}

Equation \eqref{eq:helm} has been solved for different grid sizes with  $\maxrank=50$. The time and memory consumption are reported in Table~\ref{tab:pgmres}. The storage needed for the solution scales linearly with $n$.  In addition, the table contains the number of iterations needed by the preconditioned GMRES to reach the relative tolerance $\mathsf{tol}=10^{-4}$. We note that the number of iterations grows very slowly as the grid size increases. The time required depends on many factors, such as  the distribution of the ranks and the complexity of the structure in the basis generated by the GMRES; we just remark that it grows subquadratically for this example. In future work we plan to explore the use of restarting mechanisms and other truncation strategies in order to optimize the approach.

\begin{table}
	\centering
	\fbox{
		\begin{tabular}{c|ccccc}
			& \multicolumn{5}{c}{$\maxrank = 50$}\\ 
			$n$ & $T_{\mathrm{tot}}$ (s) & $T_{\mathrm{lyap}}$ (s) & $T_{\mathrm{adapt}}$ (s) & It.& Mem. \\ \hline
			$1024$ & $72.83$ & $42.82$ & $11.39$ & $25$&$1.72$\\ 
			$2048$ & $231.92$ & $161.35$ & $24.25$ &$26$& $3.73$\\ 
			$4096$ & $603.03$ & $362.22$ & $75.74$ &$26$& $8.18$\\ 
			$8192$ & $1773.6$ & $982.37$ & $244.54$ &$26$& $16.68$\\ 
			$16384$ & $5884$ & $3065$ & $1044.4$ & $28$&$33.3$\\ 
	\end{tabular}}\\[8pt]
	\caption{Time and storage required for solving the inhomogeneous  Helmholtz equation \eqref{eq:helm}, 
		for different values of $n$ and $\maxrank=50$.  The reported memory is measured in Megabytes (MB), and refers to the storage of the solution.}\label{tab:pgmres}
\end{table}

\section{Conclusions}\label{sec:conclusions}

In  this work, we have proposed a new format for storing matrices arising from 2D discretization
of functions which are smooth almost everywhere, with localized singularities. Low-rank decompositions, which are effective for globally smooth functions, become ineffective in this case. 
The proposed structure automatically adapts to the matrix, and requires no prior information
on the location of the singular region. The storage and complexity interpolates between
dense and low-rank matrices, based on the structure, 
with these two cases as extrema. 

We demonstrated techniques for the efficient adaptation of the structure in case of moving
singularities, with the aim of tracking time-evolution of 2D PDEs; the examples show that the
proposed techniques can effectively detect changes in the structure, and ensure the desired
level of accuracy. We developed efficient Lyapunov and 
Sylvester solvers for matrix equations with \HALR{} right-hand-side  and HODLR coefficients. This case is of particular interest, as it often arises
in discretized PDEs. 
Several numerical experiments demonstrate both the effectiveness and the flexibility of the
approach. 

The proposed format may be generalized to discretization of 
3D functions by swapping quadtrees with octrees, making 
the necessary adjustments, and choosing a suitable 
low-rank format for the blocks. 
Similar ideas to the ones presented 
in this work may be used to detect the hierarchical structure 
in an adaptive way, and to adjust the structure in time. However, 
devising an efficient Sylvester solver remains challenging. 
Despite the existence of low-rank solvers for linear systems with Kronecker structure
in the Tucker format \cite{kressner2011low} exploiting
the hierarchical structure introduces major difficulties, which 
we plan to investigate in future work. 
\bibliographystyle{plain}
\bibliography{library}

\end{document}